\documentclass[a4paper]{cas-sc}
\usepackage[numbers,sort&compress]{natbib}
\usepackage[pagewise]{lineno}
\usepackage{mathrsfs}
\usepackage{amssymb,amsmath,euscript, amsthm, amsfonts, amssymb, enumerate}
\usepackage{graphicx}
\usepackage{caption}
\usepackage{subcaption}
\usepackage{siunitx}
\usepackage{epstopdf}
\usepackage{CJK,CJKnumb}
\usepackage{float}
\usepackage{color}
\usepackage[section]{placeins}
\usepackage{url}
\usepackage{tabularx}

\def\red {\color{red}}

\newtheorem{theorem}{Theorem} 
\newtheorem{lemma}{Lemma} 
\newtheorem{remark}{Remark} 
\newtheorem{example}{Example}

\def\tsc#1{\csdef{#1}{\textsc{\lowercase{#1}}\xspace}}
\tsc{WGM}
\tsc{QE}
\tsc{EP}
\tsc{PMS}
\tsc{BEC}
\tsc{DE}

\begin{document}

	\let\WriteBookmarks\relax
	\def\floatpagepagefraction{1}
	\def\textpagefraction{.001}
	\shorttitle{A Fully Implicit Stochastic SIS Epidemic Model}
	
	\shortauthors{Zeyao Jin et al.}
	
	\title [mode = title]{Dynamic Properties of Numerical Solution for Stochastic Epidemic Models with a Fully Implicit Form}
	
	\tnotemark[1]
	
	\tnotetext[1]{ }
	\author{Zeyao Jin}[style=chinese]
	\address{College of Mathematics and System Science, Xinjiang University,  Urumqi 830017, China}
	
	\author{Zhiming Li}[style=chinese]\cormark[1]
	\ead{Corresponding author at: College of Mathematics and System Science, Xinjiang University,  Urumqi 830017, China. Email address: zmli@xju.edu.cn.}
	\author{Shenglong Chen}[style=chinese]
	
\begin{abstract}
This paper establishes a fully implicit model to investigate the spread of epidemics and studies its dynamic properties, extending the semi-implicit model in Yang et al. \cite{1} (Commun. Nonlinear Sci. Numer. Simul., 152 (2026) 109140 ).   We first employ the split-step backward Euler method to address the non-adaptability arising from backward Euler discretization and ensure the global convergence of the numerical solution.  By linearizing the model, the computational cost is reduced at the expense of a slight increase in error, effectively avoiding complex iterative solution processes. Then, we derive the existence, uniqueness, positivity, and boundedness of the numerical solution, and obtain its global convergence order. On this basis, the extinction and persistence of the disease are investigated under three different scenarios. The results indicate that environmental noise is a key factor leading to disease extinction. Finally, numerical simulations visually verify the obtained theoretical conclusions.
\end{abstract}
	
\begin{keywords}
Stochastic SIS model\sep Linearly implicit $\theta$-methods\sep Slip-step backward Euler method\sep Wiener truncation
	\end{keywords}
	
	\maketitle
	
\section{Introduction}
Epidemic models are important theoretical tools for understanding and controlling disease transmission. Deterministic models, such as the classical SIR \cite{2} and SIS \cite{3} models, describe the average evolutionary trend in the number of infected individuals through ordinary differential equations, providing an important theoretical basis for public health decision-making \cite{4}. However, actual transmission processes are often influenced by stochastic factors such as individual differences, contact randomness, and environmental disturbances \cite{5,6}. When the population size is small, or the outbreak is in its initial stage, random fluctuations may cause the actual transmission trajectory to deviate significantly from the predictions of deterministic models and may even affect whether the disease goes extinct. In recent years, stochastic epidemic models have received increasing attention, including stochastic SIR models \cite{7,8,9}, stochastic SIS models \cite{10,11,12}. Based on these stochastic models, scale-free networks \cite{13}, reaction-diffusion \cite{14}, and Markov switching \cite{15} have been further applied to characterize the complex features of transmission processes.

Although stochastic models have theoretical advantages, in most cases the corresponding stochastic differential equations are difficult to solve analytically \cite{16}, making numerical methods the primary means of analysis. For problems with strong stiffness, traditional numerical methods often require extremely small step sizes due to stability constraints, resulting in low computational efficiency. To address this issue, researchers have developed a series of numerical schemes with strong stability. Among them, linear implicit methods have been widely used in the study of ordinary differential equations \cite{17,18}, partial differential equations \cite{19,20}, stochastic differential equations \cite{16,21}, and the Allen-Cahn equation \cite{22}, demonstrating good convergence and stability. Some scholars have further applied them to the study of epidemic models. For instance, Xu et al. \cite{23} established an SIR model with a nonlinear incidence rate and impulsive vaccination, and analyzed the asymptotic behavior of the impulsive linear implicit Euler method.  Most of the aforementioned numerical methods are built upon the classical It\^{o} integral framework, which requires the integrand to be adapted, meaning that the evolution of the system depends only on historical information. 

However, in some complex biological systems, when environmental feedback, time delays, or the randomization of specific parameters are considered, non-adapted stochastic differential equations may arise \cite{24}. The future state of such systems may depend on future information, violating the fundamental causality assumption of the It\^{o} integral and posing essential difficulties for numerical solutions. For example, when discretizing stochastic differential equations using the backward Euler method, both the drift and diffusion terms contain future information, making it difficult to obtain numerical solutions. To address this problem, researchers have proposed three strategies: (a) transform the non-adapted problem into an adapted form through mathematical transformations \cite{25,26}; (b) introduce more general integral concepts, such as the Skorohod integral \cite{27,28,29,30}; and (c) develop specialized numerical schemes, such as the forward integration method \cite{28,31}.  Among these three strategies, numerical methods attract the attention of many scholars. Recently, Yang et al. \cite{1} conducted numerical analysis of stochastic SIS models proposed in \cite{32} via the linear implicit $\theta$-method and established their dynamic behaviors. Next, we briefly review the modeling process of their work. 
Suppose that the population is divided into two compartments: susceptible and infectious individuals. Let $S(t)$ and $I(t)$ be the numbers of the susceptible and infectious individuals at time $t$, respectively. Gray et al. \cite{32} proposed a stochastic SIS epidemic model below:
	\begin{equation}\label{sis}
		\begin{cases}
			dS(t) = \big[ \mu N - \beta S(t) I(t) + \gamma I(t) - \mu S(t) \big] dt - \sigma S(t) I(t) dW(t), \\
			dI(t) = \big[ \beta S(t) I(t) - (\mu + \gamma) I(t) \big] dt + \sigma S(t) I(t) dW(t),
		\end{cases}
	\end{equation}
with initial values \(S(0)+I(0)=N\), where \(N\) is the total population size in which the disease is spreading. Here, \( \gamma \) is the rate at which infected individuals recover, so \( 1/\gamma \) represents the average infectious period; \( \mu \) is the per capita death rate. The parameter \( \beta \) is the disease transmission coefficient, defined as \( \beta = \lambda/N \), where \( \lambda \) is the per capita disease contact rate. 	
The model  (\ref{sis}) can be reformulated in a more compact form
	\begin{equation}\label{dI}
		dI(t) = I(t) \left[ \big( \beta N - \mu - \gamma - \beta I(t) \big) dt + \sigma (N - I(t)) dW(t) \right].
	\end{equation}
Define a threshold   $R_0^S:= \beta N/(\mu+\gamma)-\sigma^2 N^2/(2(\mu+\gamma))$. \cite{32} demonstrated the boundedness and positivity of the analytical solution $I(t)^\star$ of the model. When $R_0^S < 1$, the disease goes to extinction in both small- and large-noise environments, and when $R_0^S > 1$, the disease persists. 	
  To investigate the numerical solution of the SIS model (\ref{dI}),  let $h$ be the increment of $t$, namely the step size, and $I_{n}$ be the numerical approximation of $I(t_n)$, $n=0,1,2,\cdots$, $t_n=nh$. Using the Wiener truncation $\xi_{n}^{h}$ defined by \cite{33} and the backward Euler method, the numerical solution of Eq. (\ref{dI}) can be expressed as
\begin{equation}\label{In1}
I_{n+1}=I_{n}+h\big(\beta N - \mu - \gamma - \beta I_{n+1}\big)I_{n+1}+\sqrt{h}\sigma I_{n+1}\big(N-I_{n+1}\big)\xi_{n}^{h}, 
\end{equation}
where 
	$
		\xi_n^h=\xi_nI_{\{|\xi_n|\leq A_h\}}+A_hI_{\{\xi_n>A_h\}}-A_hI_{\{ \xi_n<-A_h\}}
$
 with a strong inequality $\mathbb{E}(\xi_n^h-\xi_n)^2\leq h$ and $A_h=\sqrt{2|\ln h|}$, where $I_A$ is an indicator function of $A$. Yang et al. \cite{1} used the linear $\theta$-methods in the drift term of Eq. (\ref{In1}) and established the following model:
 \begin{equation}\label{Y}
 	I_{n+1} = I_n + \theta h(\beta N - \mu - \gamma - \beta I_n)I_{n+1} + (1 - \theta)h(\beta N - \mu - \gamma - \beta I_n)I_n + \sqrt{h}\sigma I_n(N - I_n)\xi_n^h,
 \end{equation}
which is referred to as a semi-implicit stochastic model in this paper because of the modified drift term. It is further proved that the numerical solution $I_n^\star$ given by Eq. (\ref{Y}) is bounded and positive.  When $R_0^S < 1$ and the noise is small, the disease tends to extinction. In a high-noise environment, the disease tends to extinction regardless of the value of $R_0^S$. If $R_0^S > 1$, the disease persists. 

In the model  (\ref{Y}), note that all existing results are based on the linear -$\theta$ method in the drift term of Eq. (\ref{In1}), ignoring the effect of the random perturbation part. In practice, environmental noise may be a key factor affecting the occurrence, spread, and outcome of diseases. However,  there is relatively limited research on extending the linear implicit $\theta$-method to the diffusion term and on constructing fully implicit stochastic epidemic models. To solve the problem, we will use the Split-Step Backward Euler (SSBE) method, which offers unique advantages for handling certain types of non-adaptively or stiffly driven stochastic differential equations \cite{34}. It can reduce the difficulty of solving nonlinear systems while maintaining good stability, and exhibits strong fault tolerance to non-adapted terms, thereby yielding more reliable long-term simulation results.  Based on the above analysis, this paper discretizes the stochastic SIS epidemic model using the SSBE method and further introduces the linear implicit $\theta$-method into the diffusion term, aiming to establish a class of fully implicit stochastic SIS models and to analyze their numerical properties. The main contributions and innovations of this paper are summarized as follows:

(i)  Based on the SSBE method and the linear implicit $\theta$-methods, a class of fully implicit $\theta$ SIS epidemic models is established, in which the forward Euler method and the backward Euler method become special cases of the model.

(ii) The existence, uniqueness, positivity, and boundedness of the numerical solution are investigated, and the extinction and persistence of the disease under three different types of conditions are proved.

(iii) By linearizing the model, the computational cost is reduced with a small increase in error while ensuring the global convergence of the numerical solution.
	
The remainder of this paper is organized as follows. Section \ref{s2} proposes a novel stochastic SIS model via the SSBE method, and proves the positivity and boundness of numerical solutions.   Section \ref{sec3} provides a theoretical analysis of the model's dynamical properties, and Section \ref{s4} validates the theoretical results through numerical simulations. Finally, a brief conclusion is given in Section \ref{s5}.
	
\section{A Fully Implicit Stochastic SIS model}\label{s2}

 To obtain a more flexible form of the SIS model of (\ref{Y}), it is necessary to apply the linear implicit $\theta$-method to the diffusion term,  which is the main objective of this work. From Eq. (\ref{In1}), it can be observed that after discretizing Eq. (\ref{dI}), the stochastic term contains the term $I_{n+1}$, which implies that the stochastic term involves future information, thereby causing difficulties in solving the equation. In \cite{1}, $I_{n+1}$ is replaced by $I_{n}$ to avoid this situation. Next, we adopt a more robust approach to circumvent this difficulty.

Based on \cite{34}, we introduce the SSBE method under the local Lipschitz condition. Consider a stochastic differential equation of the form
	\begin{equation}\label{eq:sde}
		\mathrm{d}I(t) = f(I(t))\,\mathrm{d}t + g(I(t))\,\mathrm{d}W(t),\quad I(0)=I_0,\quad t\in[0,T],
	\end{equation}
	where $f:\mathbb{R}^m\to\mathbb{R}^m$, $g:\mathbb{R}^m\to\mathbb{R}^{m\times d}$, and $W(t)$ is a $d$-dimensional Wiener process defined on a filtered probability space $(\Omega,\mathcal{F},\{\mathcal{F}_t\}_{t\geq0},\mathbb{P})$. 
	Assume that the solution $I(t)$ remains almost surely in a bounded domain $\mathcal{D}\subset\mathbb{R}^m$. 
For the purposes of this definition, we consider the one-dimensional case $m=1$ with $\mathcal{D}=[0,N]$, where $N>0$ is a constant. Let $h>0$ be a time step, and define $t_n=nh$ for $n=0,1,\dots,n^{max}$ with $n^{max}h=T$. The SSBE method generates approximations $I_n\approx I(t_n)$ by the following two-step procedure:
	\begin{subequations}\label{eq:ssbe}
		\begin{align}
			I_{n+1}^* &= I_{n+1} + h\, f(I_{n+1}^*), \label{eq:ssbe1}\\
			I_{n+1} &= I_{n+1}^* + g(I_{n+1}^*)\,\Delta W_n, \label{eq:ssbe2}
		\end{align}
	\end{subequations}
	where $\Delta W_n = W(t_{n+1})-W(t_n)$ are independent Gaussian increments with $\mathbb{E}[\Delta W_n]=0$ and $\mathbb{E}[|\Delta W_n|^2]=h$, and $I_0=i_0\in[0,N]$. 	The following conditions are required for the well-posedness and convergence of the SSBE method on the domain $[0,N]$:

(i) The coefficients $f$ and $g$ are continuously differentiable on $[0,N]$, i.e., $f,g\in C^1$.
		
(ii) There exists a constant $p_1>0$ such that for all $a,b\in[0,N]$,
			$\langle a-b,\; f(a)-f(b) \rangle \leq \mu\,|a-b|^2.$

(iii) There exists a constant $p_2>0$ such that for all $a,b\in[0,N]$,
			$|g(a)-g(b)|^2 \leq p_2\,|a-b|^2.$

(iv) Let $p_3 = \bigl(p_1+\tfrac12\bigr)\vee 2p_2$. The time step $h$ must satisfy
			$h <1/(2p_3).$

	Under Conditions (i)--(iv), the SSBE method is strongly convergent in the mean-square sense, i.e.,
$$
		\lim_{h\to0}\mathbb{E}\left[\sup_{0\leq t\leq T} |\overline{I}(t)-I (t)|^2\right]=0,
$$
	where $\overline{I}(t)$ is a continuous-time extension of $\{I_n\}$.

 \begin{theorem}\label{thm1}{\rm
 	For $I(t)\in[0,N]$, Eq. (\ref{dI}) can be expressed as }
 	\begin{equation}\label{ssbe}
 		\begin{cases} 
 			I_{n+1}^*=I_n+h(\beta N-\mu-\gamma-\beta I_{n+1}^*)I_{n+1}^*,\\
 			I_{n+1}=I_{n+1}^*+\sigma I_{n+1}^*(N-I_{n+1}^*)\varDelta W_n.
 		\end{cases}
 	\end{equation}
 \end{theorem}
 \begin{proof}
 	From Eq.  (\ref{dI}), for  $0\leq t\leq T$ we have
 	$$ 
 	\begin{aligned}
 		dI(t) = I(t)[( \beta N - \mu - \gamma - \beta I(t)) dt + \sigma (N - I(t)) dW(t)]
 		:=f(I(t))dt+g(I(t))dW(t).
 	\end{aligned}
 	$$
Differentiating $f$ and $g$ with respect to the first-order derivative yields that
 	$$
 	\begin{aligned}
 		f'(I(t))=\beta N - \mu - \gamma - 2\beta I(t),\quad
 		g'(I(t))=\sigma N-2 I(t).
 	\end{aligned}
 	$$
 	Obviously, $f,g\in C^1$ and there exist constants $p_1=2\beta,p_2=8N^2+3\sigma^2N^2 >0$ such that
 	$$
 	\begin{aligned}
 		\langle a-b,f(a)-f(b)\rangle=&\big(a-b\big)\big[2\beta(b-a)\big]
 		=-2\beta(a-b)^2
 		\leq p_1|a-b|^2,\\
 		|g(a)-g(b)|^2=&\big[b^2-a^2+\sigma N(a-b)\big]^2
 		\leq2(b^2-a^2)^2+2[\sigma N(a-b)]^2\\
 		\leq&(8N^2+2\sigma^2N^2)|a-b|^2\leq p_2|a-b|^2
 	\end{aligned}
 	$$
 	for all $a,b\in[0,N]$. Based on $p_1$ and $p_2$, the selection of the step size $h$ is straightforward.
 \end{proof}
 \par 
 By combining Eq. (\ref{ssbe}) with the Wiener truncation and the linear implicit $\theta$-methods given in \cite{1}, we obtain
	\begin{equation}\label{ssbe2}
		\begin{cases} 
			I_{n+1}^*=I_n+\theta h(\beta N-\mu-\gamma-\beta I_{n+1}^*)I_{n+1}^*+(1-\theta)h(\beta N-\mu-\gamma-\beta I_n)I_{n},\\
			I_{n+1}=I_{n+1}^*+\sqrt{h}\sigma\big[\vartheta I_{n+1}^*(N-I_{n+1}^*)+(1-\vartheta)I_n(N-I_n)\big]\xi_n^h.
		\end{cases}
	\end{equation}
Clearly, both the drift term and the diffusion term in Eq. (\ref{ssbe2}) contain $(I_{n+1}^*)^2$, resulting in a nonlinear structure that requires an iterative solution and is computationally expensive. By replacing one $I_{n+1}^*$ with $I_{n}^*$, a fully implicit stochastic SIS model is established by 
	\begin{equation}\label{In4}
		\begin{cases} 
			I_{n+1}^*=I_n+\theta h(\beta N-\mu-\gamma-\beta I_n)I_{n+1}^*+(1-\theta)h(\beta N-\mu-\gamma-\beta I_n)I_{n},\\
			I_{n+1}=I_{n+1}^*+\sqrt{h}\sigma\big[\vartheta I_n(N-I_{n+1}^*)+(1-\vartheta)I_n(N-I_n)\big]\xi_n^h,\quad \theta,\vartheta\in [0,1].
		\end{cases}
	\end{equation}

 \begin{remark}  When $\theta=0$ and $\vartheta=0$, Eq. (\ref{In4}) is the forward Euler method with truncation $\xi_n^h$, which is consistent with the special case $\theta=0$ in \cite{5}. When $\theta=1$ and $\vartheta=1$, it is the backward Euler method with truncation $\xi_n^h$.
\end{remark}
Following an argument similar to that for the analytic solution, it can be shown that the numerical solution generated by Eq. (\ref{In4}) is also bounded.
	\begin{theorem}\label{thm2}{\rm
		There exists a constant \( h_0 > 0 \) such that for all \( h \) satisfying \( 0 < h < h_0 \), the numerical solution of Eq. (\ref{In4}) is positive and bounded.}
	\end{theorem}

		\begin{proof}
			Assume $I_n\in(0,N)$, solving Eq. (\ref{In4}) yields
			$$
			I_{n+1}^*=\frac{1+(1-\theta)h(\beta N-\mu-\gamma-\beta I_n)}{1-\theta h(\beta N-\mu-\gamma-\beta I_n)}I_n.
			$$
Obviously, there exists $h_1>0$ such that  $1+(1-\theta)h(\beta N-\mu-\gamma-\beta I_n)>0$ for $0<h<h_1$,  and  $h_2>0$  satisfying  $1-\theta h(\beta N-\mu-\gamma-\beta I_n)>0$ for $0<h<h_2$. Thus, for $0<h<h_1\wedge h_2$,  it follows that $I_{n+1}^*>0$ and
			$$
			\begin{aligned}
				I_{n+1}^*-N=\frac{h(\beta N-\mu-\gamma-\beta I_n)}{1-\theta h(\beta N-\mu-\gamma-\beta I_n)}+(I_n-N).
			\end{aligned}
			$$
There exists $h_3>0$ with $|\frac{h(\beta N-\mu-\gamma-\beta I_n)}{1-\theta h(\beta N-\mu-\gamma-\beta I_n)}|<| I_n-N|$ if $0<h<h_3$. Thus, $I_{n+1}^*-N<0$ $i.e.$ $I_{n+1}^*\in(0,N)$. From the definition of $\xi_{n}^h$, it follows that $-\sqrt{2|\ln h|}<\xi_{n}^h<\sqrt{2|\ln h|}$. Hence,
			$$
			\begin{aligned}
				I_{n+1}=&I_{n+1}^*+\sqrt{h}\sigma\big[\vartheta I_n(N-I_{n+1}^*)+(1-\vartheta)I_n(N-I_n)\big]\xi_n^h\\
				\geq&I_{n+1}^*-\sqrt{2h|\ln h|}\sigma\big[\vartheta I_n(N-I_{n+1}^*)+(1-\vartheta)I_n(N-I_n)\big]. 
			\end{aligned}
			$$
Since $\lim_{h\rightarrow0}{\sqrt{2h|\ln h|}}=0$, there exists $h_4>0$ for $0<h<h_4$ such that
			$$
			I_{n+1}^*-\sqrt{2h|\ln h|}\sigma\big[\vartheta I_n(N-I_{n+1}^*)+(1-\vartheta)I_n(N-I_n)\big]>0.
			$$
i.e.	$ I_{n+1}>0.$ By a simple calculation, it follows that
			$$
			\begin{aligned}
				I_{n+1}-N=&(I_{n+1}^*-N)+\sqrt{h}\sigma\big[\vartheta I_n(N-I_{n+1}^*)+(1-\vartheta)I_n(N-I_n)\big]\xi_n^h\\
				\leq &(I_{n+1}^*-N)+\sqrt{2h|\ln h|}\sigma\big[\vartheta I_n(N-I_{n+1}^*)+(1-\vartheta)I_n(N-I_n)\big].
			\end{aligned}
			$$
There  exists $h_5>0$ for $0<h<h_5$, $I_{n+1}-N<0$. 	Building upon the preceding analysis, let $ h_0=\min\{h_1,\cdots,h_5\}$ $s.t.$ $I_{n+1}\in(0,N)$ for $0<h<h_0$.
		\end{proof}
	\section{Relevant Properties of Numerical Solutions}\label{sec3}
	In this section, we investigate dynamic properties of the numerical solution generated by Eq. (\ref{In4}), including its global convergence properties, and the extinction and persistence of the disease under different conditions.

	\subsection{Convergence in Mean-square}
	Based on the positivity and boundedness of the numerical solution, we prove that it converges to the analytic solution in the mean-square sense. First, we present the following two lemmas.
	\begin{lemma}\label{lem1}{\rm
		There exist $C_1 = (\beta N + \mu + \gamma)N$ and $C_2 = (2T(\beta N + \mu + \gamma)^2 + 4\sigma^2 N^2)N^2$, independent of \( \omega \), such that}
		\[
		\mathbb{E} | (I(t) - I(s)) | < C_1 (t - s),\quad
		\mathbb{E} | I(t) - I(s) |^2 < C_2 (t - s),\quad 0 < s < t < T.
		\]
		\end{lemma}
\begin{proof}
The proof is similar to that in \cite{1} and is omitted.
\end{proof}
	
	Let \( I(t_{n+1}) \) be the analytic solution of Eq. (\ref{dI}) at point \( t_{n+1} \), and let $\widetilde{I_{n+1}}$ denote the one-step approximation, $i.e.$ $I_i=I(t_i), i=1,\cdots,n$. 
	\begin{lemma}\label{lem2}{\rm
Suppose there exist real numbers $p_1$ and $p_2$ such that
		$$
		\begin{aligned}
			 \mathbb{E}|I(t_{n+1})-\widetilde{I_{n+1}}|\leq O(h^{p_1}), \big[\mathbb{E}|I(t_{n+1})-\widetilde{I_{n+1}}|^2\big]^{\frac{1}{2}}\leq O(h^{p_2}),\quad 0\leq t_{n+1}\leq T,
		\end{aligned}
		$$
		where $p_1\geq 1/2$, $p_1\geq p_2+1/2$. Then, the global mean-square convergence order of the numerical method constructed from the one-step approximation is $ p_2- 1/2$.}
\end{lemma}
\begin{proof}
The proof is similar to that in \cite{35} and is omitted. 
\end{proof}


	\begin{theorem}\label{thm3}{\rm
		The numerical solution converges in the mean-square sense to the analytic solution generated by Eq. (\ref{dI}) on the interval [0, T], with a global convergence order of $1/2$.}
	\end{theorem}
	\begin{proof}
		By Eq. (\ref{dI}), we have
		$$
		I(t_{n+1})=I(t_{n})+\int_{t_n}^{t_{n+1}}(\beta N-\mu-\gamma-\beta I(r))I(r)dr+\int_{t_n}^{t_{n+1}}\sigma(N-I(r))I(r)dW(r).
		$$
		Define the local truncation error as
		$$
		\begin{aligned}
			&I(t_{n+1})-\widetilde{I_{n+1}}\\
			=&\frac{1}{1-\theta h(\beta N-\mu-\gamma-\beta I(t_n))}\left[\theta\int_{t_n}^{t_{n+1}}(\beta N-\mu-\gamma-\beta I(r))I(r)-(\beta N-\mu-\gamma-\beta I(t_n))I(t_{n+1})dr\right.\\
			&+(1-\theta)\int_{t_n}^{t_{n+1}}(\beta N-\mu-\gamma-\beta I(r))I(r)-(\beta N-\mu-\gamma-\beta I(t_n))I(t_{n})dr\\
			&+\vartheta\int_{t_n}^{t_{n+1}}\sigma(N-I(r))I(r)-\sigma(N-I(t_{n+1}))I(t_n)dW(r)\\
			&+(1-\vartheta)\int_{t_n}^{t_{n+1}}\sigma(N-I(r))I(r)-\sigma(N-I(t_{n}))I(t_n)dW(r)\\
			&\left.+\vartheta\sigma(N-I(t_{n+1}))I(t_{n})\sqrt{h}(\xi_{n}^h-\xi_{n}) +(1-\vartheta)\sigma(N-I(t_{n}))I(t_{n})\sqrt{h}(\xi_{n}^h-\xi_{n})\right]\\
			:=&\frac{1}{1-\theta h(\beta N-\mu-\gamma-\beta I(t_n))}\cdot\big[W_1+W_2+W_3\big],
		\end{aligned}
		$$
		where $W_1$ denotes the error arising from the drift term, $W_2$ represents the error generated by the stochastic term, $W_3$ corresponds to the truncation error, and $\xi_{n}$ is a random variable following a normal distribution.
		$$
		\begin{aligned}
			\mathbb{E}|I(t_{n+1})-\widetilde{I_{n+1}}|
			=&\mathbb{E}|\frac{1}{1-\theta h(\beta N-\mu-\gamma-\beta I(t_n))}\cdot\big[W_1+W_2+W_3\big]|\\
			\leq&|\frac{1}{1-\theta h(\beta N-\mu-\gamma-\beta I(t_n))}|\cdot\mathbb{E}(|W_1|+|W_2|+|W_3|)	\\
			=&|\frac{1}{1-\theta h(\beta N-\mu-\gamma-\beta I(t_n))}|\cdot\mathbb{E}|W_1|	\end{aligned}
		$$$$
		\begin{aligned}
			\leq&|\frac{1}{1-\theta h(\beta N-\mu-\gamma-\beta I(t_n))}|\cdot\mathbb{E}|\theta\int_{t_n}^{t_{n+1}}(\beta N-\mu-\gamma-\beta I(r))I(r)-(\beta N-\mu-\gamma-\beta I(t_n))I(t_{n+1})dr\\
			&+(1-\theta)\int_{t_n}^{t_{n+1}}(\beta N-\mu-\gamma-\beta I(r))I(r)-(\beta N-\mu-\gamma-\beta I(t_n))I(t_{n})dr|.
		\end{aligned}
		$$
		Consider the $\theta$ term:
		$$
		\begin{aligned}
			&\mathbb{E}|\theta\int_{t_n}^{t_{n+1}}(\beta N-\mu-\gamma-\beta I(r))I(r)-(\beta N-\mu-\gamma-\beta I(t_n))I(t_{n+1})dr|\\
			\leq&\theta\int_{t_n}^{t_{n+1}}\mathbb{E}|(\beta N-\mu-\gamma-\beta I(r))I(r)-(\beta N-\mu-\gamma-\beta I(t_n))I(t_{n+1})|dr\\
			=&\theta\int_{t_n}^{t_{n+1}}\mathbb{E}|(\beta N-\mu-\gamma-\beta I(r))I(r)-(\beta N-\mu-\gamma-\beta I(t_n))I(r)\\
			&+(\beta N-\mu-\gamma-\beta I(t_n))I(r)-(\beta N-\mu-\gamma-\beta I(t_n))I(t_{n+1})|dr\\
			\leq&\theta\int_{t_n}^{t_{n+1}}\mathbb{E}\big[\beta|I(r)(I(t_n)-I(r))|+|(\beta N-\mu-\gamma-\beta I(t_n))(I(r)-I(t_n))|\big]dr,
		\end{aligned}
		$$
and  the 1-$\theta$ term
		$$
		\begin{aligned}
			&\mathbb{E}|(1-\theta)\int_{t_n}^{t_{n+1}}(\beta N-\mu-\gamma-\beta I(r))I(r)-(\beta N-\mu-\gamma-\beta I(t_n))I(t_{n})dr|\\
			\leq&(1-\theta)\int_{t_n}^{t_{n+1}}\mathbb{E}|(\beta N-\mu-\gamma-\beta I(r))I(r)-(\beta N-\mu-\gamma-\beta I(t_n))I(t_{n})|dr\\
			=&(1-\theta)\int_{t_n}^{t_{n+1}}\mathbb{E}|(\beta N-\mu-\gamma-\beta I(r))I(r)-(\beta N-\mu-\gamma-\beta I(r))I(t_n)\\
			&+(\beta N-\mu-\gamma-\beta I(r))I(t_n)-(\beta N-\mu-\gamma-\beta I(t_n))I(t_{n})|dr\\
			\leq&(1-\theta)\int_{t_n}^{t_{n+1}}\mathbb{E}\big[|(\beta N-\mu-\gamma-\beta I(r))(I(r)-I(t_n))|+|(\beta I(t_n)-\beta I(r))I(t_n)|\big]dr.
		\end{aligned}
		$$
		Based on Lemma 1, we obtain $\mathbb{E}|I(t_{n+1})-\widetilde{I_{n+1}}|\leq O(h^2)$. Then,
		$$
		\begin{aligned}
			\mathbb{E}|I(t_{n+1})-\widetilde{I_{n+1}}|^2
			\leq&|\frac{1}{1-\theta h(\beta N-\mu-\gamma-\beta I(t_n))}|^2\cdot\mathbb{E}|3W_1^2+3W_2^2+3W_3^2|\\
			\leq&3|\frac{1}{1-\theta h(\beta N-\mu-\gamma-\beta I(t_n))}|^2\big[\mathbb{E}|W_1^2|+\mathbb{E}|W_2^2|+\mathbb{E}|W_3^2|\big].
		\end{aligned}
		$$
		Considering the component $\mathbb{E}|W_1^2|$, from the inequality $\mathbb{E}|\theta A\pm(1-\theta)B|^2\leq\mathbb{E}|2\theta^2 A^2+2(1-\theta)^2B^2|$, we have
		$$
		\begin{aligned}
			\mathbb{E}|W_1^2|
			=&\mathbb{E}|\theta\int_{t_n}^{t_{n+1}}(\beta N-\mu-\gamma-\beta I(r))I(r)-(\beta N-\mu-\gamma-\beta I(t_n))I(t_{n+1})dr\\
			&+(1-\theta)\int_{t_n}^{t_{n+1}}(\beta N-\mu-\gamma-\beta I(r))I(r)-(\beta N-\mu-\gamma-\beta I(t_n))I(t_{n})dr|^2		\\
			\leq&2\theta^2\mathbb{E}|\int_{t_n}^{t_{n+1}}(\beta N-\mu-\gamma-\beta I(r))I(r)-(\beta N-\mu-\gamma-\beta I(t_n))I(t_{n+1})dr|^2\\
			&+2(1-\theta)^2\mathbb{E}|\int_{t_n}^{t_{n+1}}(\beta N-\mu-\gamma-\beta I(r))I(r)-(\beta N-\mu-\gamma-\beta I(t_n))I(t_{n})dr|^2\\
			\leq&2\theta^2\mathbb{E}\big[h\int_{t_n}^{t_{n+1}}|(\beta N-\mu-\gamma-\beta I(r))I(r)-(\beta N-\mu-\gamma-\beta I(t_n))I(t_{n+1})|^2dr\big]\\
			&+2(1-\theta)^2\mathbb{E}\big[h\int_{t_n}^{t_{n+1}}|(\beta N-\mu-\gamma-\beta I(r))I(r)-(\beta N-\mu-\gamma-\beta I(t_n))I(t_{n})|^2dr\big]\\\end{aligned}
		$$$$
		\begin{aligned}
			=&2\theta^2h\mathbb{E}\big[\int_{t_n}^{t_{n+1}}|\beta I(r)(I(t_n)-I(r))+(\beta N-\mu-\gamma-\beta I(t_n))(I(r)-I(t_{n+1}))|^2dr\big]\\
			&+2(1-\theta)^2h\mathbb{E}\big[\int_{t_n}^{t_{n+1}}|(\beta N-\mu-\gamma-\beta I(r))(I(r)-I(t_{n}))+\beta I(t_n)(I(t_n)-I(r))|^2dr\big]		\\
			\leq&4\theta^2h\mathbb{E}\big[\int_{t_n}^{t_{n+1}}|\beta I(r)(I(t_n)-I(r))|^2+|(\beta N-\mu-\gamma-\beta I(t_n))(I(r)-I(t_{n+1}))|^2dr\big]\\
			&+4(1-\theta)^2h\mathbb{E}\big[\int_{t_n}^{t_{n+1}}|(\beta N-\mu-\gamma-\beta I(r))(I(r)-I(t_{n}))|^2+|\beta I(t_n)(I(t_n)-I(r))|^2dr\big].
		\end{aligned}
		$$
		By Lemma \ref{lem1}, we obtain $\mathbb{E}|W_1^2|\leq O(h^3)$. According to Itô's formula, one has 
		$$
		\begin{aligned}
			\mathbb{E}|W_2^2|
			=&\mathbb{E}\big[|\vartheta\int_{t_n}^{t_{n+1}}\sigma(N-I(r))I(r)-\sigma(N-I(t_{n+1}))I(t_n)dW(r)\\
			&+(1-\vartheta)\int_{t_n}^{t_{n+1}}\sigma(N-I(r))I(r)-\sigma(N-I(t_{n}))I(t_n)dW(r)|^2\big]\\
			\leq&2\mathbb{E}\big[|\vartheta\int_{t_n}^{t_{n+1}}\sigma(N-I(r))I(r)-\sigma(N-I(t_{n+1}))I(t_n)dW(r)|^2\\
			&+|(1-\vartheta)\int_{t_n}^{t_{n+1}}\sigma(N-I(r))I(r)-\sigma(N-I(t_{n}))I(t_n)dW(r)|^2\big]\\
			=&2\vartheta^2\mathbb{E}\big[\int_{t_n}^{t_{n+1}}|\sigma(N-I(r))I(r)-\sigma(N-I(t_{n+1}))I(t_n)|^2dr\big]\\
			&+2(1-\vartheta)^2\big[\int_{t_n}^{t_{n+1}}|\sigma(N-I(r))I(r)-\sigma(N-I(t_{n}))I(t_n)|^2dr\big]\\
			=&2\vartheta^2\mathbb{E}\big[\int_{t_n}^{t_{n+1}}|\sigma(N-I(r))(I(r)-I(t_n))+\sigma I(t_n)(I(t_{n+1})-I(r))|^2dr\big]\\
			&+2(1-\vartheta)^2\big[\int_{t_n}^{t_{n+1}}|\sigma(N-I(r))(I(r)-I(t_n))+\sigma I(t_n)(I(t_n)-I(r))|^2dr\big]\\
			\leq&2\vartheta^2\mathbb{E}\big[2\int_{t_n}^{t_{n+1}}|\sigma(N-I(r))(I(r)-I(t_n))|^2+|\sigma I(t_n)(I(t_{n+1})-I(r))|^2dr\big]\\
			&+2(1-\vartheta)^2\big[2\int_{t_n}^{t_{n+1}}|\sigma(N-I(r))(I(r)-I(t_n))|^2+|\sigma I(t_n)(I(t_n)-I(r))|^2dr\big].
		\end{aligned}
		$$
		Obviously $\mathbb{E}|W_2^2|\leq O(h^2)$. Now, we consider the final part
		$$
		\begin{aligned}
			\mathbb{E}|W_3^2|
			=&\mathbb{E}|\vartheta\sigma(N-I(t_{n+1}))I(t_{n})\sqrt{h}(\xi_{n}^h-\xi_{n})
			+(1-\vartheta)\sigma(N-I(t_{n}))I(t_{n})\sqrt{h}(\xi_{n}^h-\xi_{n})|^2\\
			\leq&\mathbb{E}|2\sigma^2N^4h(\xi_{n}^h-\xi_{n})^2|+\mathbb{E}|2\sigma^2N^4h(\xi_{n}^h-\xi_{n})^2|\\
			=&4\sigma^2N^4h\mathbb{E}(\xi_{n}^h-\xi_{n})^2
			\leq O(h^2).
		\end{aligned}
		$$
		In summary,  
		$
			\mathbb{E} |I(t_{n+1})-\widetilde{I_{n+1}}|^2\leq O(h^2)$,  i.e. 
			$[\mathbb{E}|I(t_{n+1})-\widetilde{I_{n+1}}|^2]^{\frac{1}{2}}\leq O(h).
		$
	Thus, 	$\mathbb{E}|I(t_{n+1})-\widetilde{I_{n+1}}|$ and $[\mathbb{E}|I(t_{n+1})-\widetilde{I_{n+1}}|^2]^{\frac{1}{2}}$ satisfying Lemma \ref{lem2}, the global convergence order of the numerical solution is $1/2$.
	\end{proof}
	\subsection{Persistence and Extinction}\label{s3}
	The persistence and extinction of the disease are two crucial issues in infectious disease modeling. \cite{32} has thoroughly investigated the dynamic behavior of the analytical solution for model (\ref{dI}). Next, we further examine the dynamic behavior of the numerical solution derived from Eq. (\ref{In4}).\par 
Based on \cite{1}, the following localized stochastic stability function \( R(I_n) \) is defined, and
	$I_{n+1}=R(I_n)I_n,$
	where
	$$
	R(I_n)=\frac{I_{n+1}^*}{I_n}+\sqrt{h}\sigma\big[\vartheta(N-I_{n+1}^*)+(1-\vartheta)(N-I_n)\big]\xi_n^h.
	$$
	For convenience, we denote \( R(I_n) \) as \( R_n \).

	\begin{lemma} \label{lem3}{\rm(\cite{1}) 
		Let \( p(x)\) be the density function of the standard normal distribution. For any sufficiently small \( h>0 \), we have  }
		$$
		\begin{aligned}
	 \int_{A_h}^{\infty} p(x)dx  = O(h), 
 \int_{A_h}^{\infty} x^2 p(x)dx = O(h^{1-\frac{1}{2}\varepsilon}), 
 \int_{A_h}^{\infty} x^4 p(x)dx = O(h^{1-\frac{3}{2}\varepsilon}).
		\end{aligned}
		$$
	\end{lemma}
	\begin{lemma}\label{lem4}{\rm
		For a sufficiently small $h$, the conditional expectation of the function $R_n$ satisfies:
		$$
		\begin{aligned}
			\mathbb{E}\big(\log{R_n}|\mathcal{F}_{t_n}\big)=&\frac{2(\beta N-\mu-\gamma-\beta I_n)-\sigma^2\big[\vartheta(N-I_{n+1}^*)-(1-\vartheta)(N-I_n)\big]^2}{2R(h,A_h)[1-\theta h(\beta N-\mu-\gamma-\beta I_n)]^2}h+O(h^2),
		\end{aligned}
		$$
		where }
		$$
		\begin{aligned}
			R(h,x):=\big(\frac{1+(1-\theta)h(\beta N-\mu-\gamma-\beta I_n)}{1-\theta h(\beta N-\mu-\gamma-\beta I_n)}\big)^2-\big[\sqrt{h}\sigma\vartheta(N-I_{n+1}^*)x+\sqrt{h}\sigma(1-\vartheta)(N-I_n)x\big]^2.
		\end{aligned}
		$$
	\end{lemma}
	\begin{proof}
		Let
		$$
		\begin{aligned}
			Q=&\int_{-\infty}^{+\infty}\log\big(\frac{1+(1-\theta)h(\beta N-\mu-\gamma-\beta I_n)}{1-\theta h(\beta N-\mu-\gamma-\beta I_n)}+\sqrt{h}\sigma\vartheta(N-I_{n+1}^*)x+\sqrt{h}\sigma(1-\vartheta)(N-I_n)x\big)p(x)dx,
		\end{aligned}
		$$
		where $x=\xi_n^h$. Substituting \( x \) with \( -x \) yields
		$$
		\begin{aligned}
			Q=&\int_{+\infty}^{-\infty}\log\big(\frac{1+(1-\theta)h(\beta N-\mu-\gamma-\beta I_n)}{1-\theta h(\beta N-\mu-\gamma-\beta I_n)}-\sqrt{h}\sigma\vartheta(N-I_{n+1}^*)x -\sqrt{h}\sigma(1-\vartheta)(N-I_n)x\big)p(-x)d(-x)\\
			=&\int_{-\infty}^{+\infty}\log\big(\frac{1+(1-\theta)h(\beta N-\mu-\gamma-\beta I_n)}{1-\theta h(\beta N-\mu-\gamma-\beta I_n)}-\sqrt{h}\sigma\vartheta(N-I_{n+1}^*)x -\sqrt{h}\sigma(1-\vartheta)(N-I_n)x\big)p(x)dx.
		\end{aligned}
		$$	
		Adding the two equations above yields
		$$
		\begin{aligned}
			2Q=&\int_{-\infty}^{+\infty}\big[\log\big(\frac{1+(1-\theta)h(\beta N-\mu-\gamma-\beta I_n)}{1-\theta h(\beta N-\mu-\gamma-\beta I_n)}+\sqrt{h}\sigma\vartheta(N-I_{n+1}^*)x\\
			&+\sqrt{h}\sigma(1-\vartheta)(N-I_n)x\big)\log\big(\frac{1+(1-\theta)h(\beta N-\mu-\gamma-\beta I_n)}{1-\theta h(\beta N-\mu-\gamma-\beta I_n)}\\
			&-\sqrt{h}\sigma\vartheta(N-I_{n+1}^*)x-\sqrt{h}\sigma(1-\vartheta)(N-I_n)x\big)\big]p(x)dx\\
			=&\int_{-\infty}^{+\infty}\log\big[\big(\frac{1+(1-\theta)h(\beta N-\mu-\gamma-\beta I_n)}{1-\theta h(\beta N-\mu-\gamma-\beta I_n)}\big)^2\\
			&-\big(\sqrt{h}\sigma\vartheta(N-I_{n+1}^*)x+\sqrt{h}\sigma(1-\vartheta)(N-I_n)x\big)^2\big]p(x)dx.
		\end{aligned}
		$$
		It is evident that the exponential part is an even function, thus we have:
		$$
		\begin{aligned}
			\mathbb{E}\big(\log{R_n}|\mathcal{F}_{t_n}\big)=
			&\frac{1}{2}\int_{-\infty}^{+\infty}\log\big[\big(\frac{1+(1-\theta)h(\beta N-\mu-\gamma-\beta I_n)}{1-\theta h(\beta N-\mu-\gamma-\beta I_n)}\big)^2\\
			&-\big(\sqrt{h}\sigma\vartheta(N-I_{n+1}^*)x+\sqrt{h}\sigma(1-\vartheta)(N-I_n)x\big)^2\big]p(x)dx\\
			=&\int_{0}^{A_h}\log\big[\big(\frac{1+(1-\theta)h(\beta N-\mu-\gamma-\beta I_n)}{1-\theta h(\beta N-\mu-\gamma-\beta I_n)}\big)^2\\
			&-\big(\sqrt{h}\sigma\vartheta(N-I_{n+1}^*)x+\sqrt{h}\sigma(1-\vartheta)(N-I_n)x\big)^2\big]p(x)dx.\\
			&+\int_{A_h}^{+\infty}\log\big[\big(\frac{1+(1-\theta)h(\beta N-\mu-\gamma-\beta I_n)}{1-\theta h(\beta N-\mu-\gamma-\beta I_n)}\big)^2\\
			&-\big(\sqrt{h}\sigma\vartheta(N-I_{n+1}^*)A_h+\sqrt{h}\sigma(1-\vartheta)(N-I_n)A_h\big)^2\big]p(x)dx.
		\end{aligned}
		$$
		From the definition of \( R(h,x) \) and Lemma 3.4 in \cite{1}, it follows that
		$$
		\begin{aligned}
			\mathbb{E}\big(\log{R_n}|\mathcal{F}_{t_n}\big)
			=&\int_{0}^{A_h}\log R(h,x)p(x)dx+\int_{A_h}^{+\infty}\log R(h,A_h)p(x)dx\\
			=&\int_{0}^{A_h}\log \big(\frac{R(h,x)}{R(h,A_h)}\big)p(x)dx+\int_{0}^{+\infty}\log R(h,A_h)p(x)dx\\
			=&\frac{R(h,1)-1}{2R(h,A_h)}+\frac{1-R(h,A_h)+R(h,A_h)\log(R(h,A_h))}{2R(h,A_h)} +J(h,A_h),
		\end{aligned}
		$$
		where 
		$$
		\begin{aligned}
			J(h,A_h)=&\int_{0}^{A_h}\big[\log \big(\frac{R(h,x)}{R(h,A_h)}\big)-\big(\frac{R(h,x)-R(h,A_h)}{R(h,A_h)}\big)\big]p(x)dx -\int_{A_h}^{+\infty}\log \big(\frac{R(h,x)-R(h,A_h)}{R(h,A_h)}\big)p(x)dx.
		\end{aligned}
		$$
		Next, we examine the values of the three parts above separately.\par 
		$\mathbf{Part\,I:}$ Let
		$$
		\begin{aligned}
			\frac{R(h,1)-1}{2R(h,A_h)}=\frac{(R(h,1)-1)\big[1-\theta h(\beta N-\mu-\gamma-\beta I_n)\big]^2}{2R(h,A_h)\big[1-\theta h(\beta N-\mu-\gamma-\beta I_n)\big]^2}=\frac{Z_1}{M_1}.
		\end{aligned}
		$$
		The calculation yields
		$$
		\begin{aligned}
			Z_1=&\big[1+(1-\theta)h(\beta N-\mu-\gamma-\beta I_n)\big]^2-\big[\sqrt{h}\sigma\vartheta(N-I_{n+1}^*)\\
			&+\sqrt{h}\sigma(1-\vartheta)(N-I_n)\big]^2\big[1-\theta h(\beta N-\mu-\gamma-\beta I_n)\big]^2 -\big[1-\theta h(\beta N-\mu-\gamma-\beta I_n)\big]^2\\
			=&\{2(\beta N-\mu-\gamma-\beta I_n)-\sigma^2\big[\vartheta(N-I_{n+1}^*)-(1-\vartheta)(N-I_n)\big]^2\}h +O(h^2),
		\end{aligned}
		$$
		and
		$$
		\begin{aligned}
			M_1=&2R(h,A_h)\big[1-\theta h(\beta N-\mu-\gamma-\beta I_n)\big]^2\\
			=&2\big[1+(1-\theta)h(\beta N-\mu-\gamma-\beta I_n)\big]^2-2\big[\sqrt{h}\sigma\vartheta(N-I_{n+1}^*)A_h\\
			&+\sqrt{h}\sigma(1-\vartheta)(N-I_n)A_h\big]^2\big[1-\theta h(\beta N-\mu-\gamma-\beta I_n)\big]^2.
		\end{aligned}
		$$
		Since \( M_1 \) contains the constant term 2, it follows that
		$$
		\begin{aligned}
			\frac{R(h,1)-1}{2R(h,A_h)}=&\frac{2(\beta N-\mu-\gamma-\beta I_n)-\sigma^2\big[\vartheta(N-I_{n+1}^*)-(1-\vartheta)(N-I_n)\big]^2}{2R(h,A_h)[1-\theta h(\beta N-\mu-\gamma-\beta I_n)]^2}h +O(h^2).  
		\end{aligned}
		$$\par 
		$\mathbf{Part\,II:}$ Considering the inequality \( y-4y^2<\ln(1+y)<y \) holds when \( |y|<\frac{1}{2} \) is satisfied. Then, we have
		$$
		\begin{aligned}
			&\frac{1-R(h,A_h)+R(h,A_h)\log(R(h,A_h))}{2R(h,A_h)}\\
			<&\frac{1-R(h,A_h)+R(h,A_h)(R(h,A_h)-1)}{2R(h,A_h)}\\
			=&\frac{(R(h,A_h)-1)^2}{2R(h,A_h)}
			=\frac{(R(h,A_h)-1)^2[1-\theta h(\beta N-\mu-\gamma-\beta I_n)]^4}{2R(h,A_h)[1-\theta h(\beta N-\mu-\gamma-\beta I_n)]^4}:=\frac{Z_2}{M_2},
		\end{aligned}
		$$
		here
		$$
		\begin{aligned}
			Z_2=&\{\big[1+(1-\theta)h(\beta N-\mu-\gamma-\beta I_n)\big]^2-\big[\sqrt{h}\sigma\vartheta(N-I_{n+1}^*)\\
			&+\sqrt{h}\sigma(1-\vartheta)(N-I_n)\big]^2\big[1-\theta h(\beta N-\mu-\gamma-\beta I_n)\big]^2 -\big[1-\theta h(\beta N-\mu-\gamma-\beta I_n)\big]^2\}^2\\
			=&\{\big[1+(1-\theta)h(\beta N-\mu-\gamma-\beta I_n)\big]^2-\big[1-\theta h(\beta N-\mu-\gamma-\beta I_n)\big]^2\\
			&-\big[\sqrt{h}\sigma\vartheta(N-I_{n+1}^*)+\sqrt{h}\sigma(1-\vartheta)(N-I_n)\big]^2 \big[1-\theta h(\beta N-\mu-\gamma-\beta I_n)\big]^2\}^2
\end{aligned}
		$$
$$
		\begin{aligned}
			=&\{1+2(1-\theta)h(\beta N-\mu-\gamma-\beta I_n)+(1-\theta)^2h^2(\beta N-\mu-\gamma-\beta I_n)^2-1+2\theta h(\beta N-\mu-\gamma-\beta I_n)^2\\
			&-\theta^2h^2(\beta N-\mu-\gamma-\beta I_n)^4-h\sigma^2A_h^2\big[\vartheta^2(N-I_{n+1}^*)^2+(1-\vartheta)^2(N-I_{n})^2\\
			&+2\vartheta(1-\vartheta)(N-I_{n+1}^*)(N-I_{n})\big]\big[1-2\theta h(\beta N-\mu-\gamma-\beta I_n) +\theta^2 h^2(\beta N-\mu-\gamma-\beta I_n)^2\big]\}^2\\
			=&\{\big[2(1-\theta)(\beta N-\mu-\gamma-\beta I_n)+2\theta (\beta N-\mu-\gamma-\beta I_n)^2\big]h\\
			&+(1-\theta)^2h^2(\beta N-\mu-\gamma-\beta I_n)^2-\theta^2h^2(\beta N-\mu-\gamma-\beta I_n)^4 +O(h)\}^2\\
			=&O(h^2),
		\end{aligned}
		$$
		and
		$ 
			M_2=2R(h,A_h)[1-\theta h(\beta N-\mu-\gamma-\beta I_n)]^4.
		 $
		Similar to \( M_1 \), it is evident that \( M_2 \) also contains the constant term 2. Consequently, it follows that
		$$
		\begin{aligned}
			\frac{1-R(h,A_h)+R(h,A_h)\log(R(h,A_h))}{2R(h,A_h)}<O(h^2).
		\end{aligned}
		$$\par 
		$\mathbf{Part\,III:}$Consider separately the two integrals within \( J(h,A_h) \).
		\par 
		$\mathbf{(i)}$ First, we consider term 
		\begin{equation}\label{11}
			\int_{0}^{A_h}\big[\log \big(\frac{R(h,x)}{R(h,A_h)}\big)-\big(\frac{R(h,x)-R(h,A_h)}{R(h,A_h)}\big)\big]p(x)dx.
		\end{equation}
		\par 
		Since inequality \( y-\frac{y^2}{2}<\log(1+y)<y \) holds for \( y>0 \), it follows that \( y-\log(1+y)<\frac{y^2}{2} \) and \(\log(1+y)-y<0\). It is evident that \( R(h,x) \) is non-increasing,   therefore, \( \frac{R(h,x)-R(h,A_h)}{R(h,A_h)}>0 \) holds on \( [0,A_h] \). Then we have
		\begin{equation}\label{12}
			\log\big(\frac{R(h,x)}{R(h,A_h)}\big)-\frac{R(h,x)-R(h,A_h)}{R(h,A_h)}<\frac{R(h,x)-R(h,A_h)}{R(h,A_h)}-\log\big(\frac{R(h,x)}{R(h,A_h)}\big)<\frac{1}{2}\big(\frac{R(h,x)-R(h,A_h)}{R(h,A_h)}\big)^2.
		\end{equation}
		Let 
		$$
		\begin{aligned}
			\frac{R(h,x)-R(h,A_h)}{R(h,A_h)}=&\frac{(R(h,x)-R(h,A_h))\big[1-\theta h(\beta N-\mu-\gamma-\beta I_n)\big]^2}{R(h,A_h)\big[1-\theta h(\beta N-\mu-\gamma-\beta I_n)\big]^2} 
			= \frac{Z_3}{M_3}.
		\end{aligned}
		$$
		Denote \( k(x) = \sqrt{2}h\vartheta(N-I_{n+1}^2)x+\sqrt{2}h(1-\vartheta)(N-I_{n})x \). Calculation yields
		$$
		\begin{aligned}
			Z_3=&\big[1-\theta h(\beta N-\mu-\gamma-\beta I_n)\big]^2\big[k(A_h)^2-k(x)^2\big]\\
			=&\big[1-\theta h(\beta N-\mu-\gamma-\beta I_n)\big]^2 (A_h^2-x^2)\big[\sigma\vartheta(N-I_{n+1}^*)+\sigma(1-\vartheta)(N-I_{n})\big]^2h
		\end{aligned}
		$$
		and
		$M_3=\frac{1}{2}M_1. $
		From \( M_1 \), \( M_3 \) also contains the constant term 1. Substituting Eq. (\ref{12}) into Eq. (\ref{11}) yields
$$
		\begin{aligned}
			&\int_{0}^{A_h}\big[\log \big(\frac{R(h,x)}{R(h,A_h)}\big)-\big(\frac{R(h,x)-R(h,A_h)}{R(h,A_h)}\big)\big]p(x)dx\\
			<&\int_{0}^{A_h}\big[\big(\frac{R(h,x)-R(h,A_h)}{R(h,A_h)}\big)-\log \big(\frac{R(h,x)}{R(h,A_h)}\big)\big]p(x)dx\\
			<&\int_{0}^{A_h}\big[\frac{1}{2}\big(\frac{R(h,x)-R(h,A_h)}{R(h,A_h)}\big)^2\big]p(x)dx\\
			=&\frac{1}{2M_3^2}\int_{0}^{A_h}\{\big[\sigma\vartheta(N-I_{n+1}^*)+\sigma(1-\vartheta)(N-I_{n})\big]^2\\
			&\times\big[h-2\theta(\beta N-\mu-\gamma-\beta I_n)h^2+\theta^2(\beta N-\mu-\gamma-\beta I_n)^2h^3\big] (A_h^2-x^2)\}^2p(x)dx\\
			=&\frac{1}{2M_3^2}\big[\sigma\vartheta(N-I_{n+1}^*)+\sigma(1-\vartheta)(N-I_{n})\big]^4\\
			&\times\big[h-2\theta(\beta N-\mu-\gamma-\beta I_n)h^2+\theta^2(\beta N-\mu-\gamma-\beta I_n)^2h^3\big]^2 \int_{0}^{A_h}(A_h^2-x^2)^2p(x)dx\end{aligned}
$$$$
		\begin{aligned}
			=&\frac{1}{2M_3^2}\big[\sigma\vartheta(N-I_{n+1}^*)+\sigma(1-\vartheta)(N-I_{n})\big]^4O(h^2)\int_{0}^{A_h}(A_h^2-x^2)^2p(x)dx\\
			=&\frac{1}{2M_3^2}\big[\sigma\vartheta(N-I_{n+1}^*)+\sigma(1-\vartheta)(N-I_{n})\big]^4O(h^2)\big[\frac{1}{2}-\int_{A_h}^{+\infty}(A_h^2-x^2)^2p(x)dx\big]\\
			=&\frac{1}{2M_3^2}\big[\sigma\vartheta(N-I_{n+1}^*)+\sigma(1-\vartheta)(N-I_{n})\big]^4O(h^2)\\
			&\times\big[\frac{1}{2}-A_h^4\int_{A_h}^{+\infty}p(x)dx+2A_h^2\int_{A_h}^{+\infty}x^2p(x)dx-\int_{A_h}^{+\infty}x^4p(x)dx\big].
		\end{aligned}
$$
		From the conclusion of Lemma \ref{lem3}, it follows that
		$$
		\begin{aligned}
			&\int_{0}^{A_h}\big[\log \big(\frac{R(h,x)}{R(h,A_h)}\big)-\big(\frac{R(h,x)-R(h,A_h)}{R(h,A_h)}\big)\big]p(x)dx\\
			<&\frac{1}{2M_3^2}\big[\sigma\vartheta(N-I_{n+1}^*)+\sigma(1-\vartheta)(N-I_{n})\big]^4O(h^2) \big[\frac{1}{2}-A_h^4\int_{A_h}^{+\infty}p(x)dx+O(h^{1-\frac{1}{2}\varepsilon})\big]
			=O(h^2).
		\end{aligned}
		$$
		\par 
		$\mathbf{(ii)}$ We consider term 
		$
		\int_{A_h}^{+\infty}\log \big(\frac{R(h,x)-R(h,A_h)}{R(h,A_h)}\big)p(x)dx.
		$
		Direct calculation yields
		$$
		\begin{aligned}
			&\int_{A_h}^{+\infty}\log \big(\frac{R(h,x)-R(h,A_h)}{R(h,A_h)}\big)p(x)dx\\
			=&\frac{1}{M_3}\int_{0}^{A_h}\big[\sigma\vartheta(N-I_{n+1}^*)+\sigma(1-\vartheta)(N-I_{n})\big]^2\\
			&\times\big[h-2\theta(\beta N-\mu-\gamma-\beta I_n)h^2+\theta^2(\beta N-\mu-\gamma-\beta I_n)^2h^3\big] (A_h^2-x^2)p(x)dx\\
			=&\frac{1}{M_3}\big[\sigma\vartheta(N-I_{n+1}^*)+\sigma(1-\vartheta)(N-I_{n})\big]^2O(h) \int_{0}^{A_h}(A_h^2-x^2)p(x)dx
			<O(h^2).
		\end{aligned}
		$$
		From $\mathbf{(i)}$ and $\mathbf{(ii)}$, it follows that \( J(h, A_h)<O(h^2) \).
		 
		From the three parts above, we have
		$$
		\begin{aligned}
			\mathbb{E}\big(\log{R_n}|\mathcal{F}_{t_n}\big)=&\frac{2(\beta N-\mu-\gamma-\beta I_n)-\sigma^2\big[\vartheta(N-I_{n+1}^*)-(1-\vartheta)(N-I_n)\big]^2}{2R(h,A_h)[1-\theta h(\beta N-\mu-\gamma-\beta I_n)]^2}h +O(h^2).
		\end{aligned}
		$$
	\end{proof}
	Based on the results of the above Lemmas, we next consider the dynamic behavior of the numerical solution according to three different scenarios: (i)  the threshold \( R_0^S < 1 \) and the noise environment is small,  (ii)   a high-noise environment,  and (iii)   \( R_0^S > 1 \).

	\begin{theorem}\label{thm4}{\rm
		If $R_0^S<1$ and $\sigma^2<\frac{\beta}{2N}$, there exists $h_0$ s.t. 
		$
		\lim_{n\rightarrow\infty} \sup \frac{\log I_{n+1}}{(n+1)h}<0
		$
		for $0<h<h_0$, $i.e.$ $I_{n+1}\rightarrow0$.}
	\end{theorem}
		\begin{proof}
			From the definition of \( R_n \), we have
			$
			\begin{aligned}
				\log I_{n+1}=\log R_n+\log I_n
				=\log R_0+\sum_{i=0}^{n}\log R_i.
			\end{aligned}
			$
			Denote \( M_n=\log R_n-\mathbb{E}(\log R_n) \) is a random variable with a mean of zero. And let $S_n=\sum_{i=0}^{n}M_i  $ is a martingale. Thus,
			$$
			\begin{aligned}
				\log I_{n+1}=&\log R_0+\sum_{i=0}^{n}\big(M_i+\mathbb{E}(\log R_i)\big)
				=\log R_0+S_n+\sum_{i=0}^{n} \mathbb{E}\big(\log R_i\big)\\
				=&\log R_0+S_n+\sum_{i=0}^{n} \mathbb{E}\big(\mathbb{E}(\log R_i\mid \mathcal{F}_{t_i})\big).
			\end{aligned}
			$$
			From Lemma \ref{lem4}, it follows that
			$$
			\begin{aligned}
				\mathbb{E}\big(\log R_i\mid \mathcal{F}_{t_i}\big)=&\frac{2(\beta N-\mu-\gamma-\beta I_n)-\sigma^2\big[\vartheta(N-I_{n+1}^*)-(1-\vartheta)(N-I_n)\big]^2}{2R(h,A_h)[1-\theta h(\beta N-\mu-\gamma-\beta I_n)]^2}h +O(h^2)
				:=f(I_n)h+O(h^2).
			\end{aligned}
			$$
			If $R_0^S<1$ and $\sigma^2<\frac{\beta}{2N}$, direct calculation yields
			$$
			\begin{aligned}
				f(I_n)=&\frac{1}{M_1}\big[2(\beta N-\mu-\gamma)-\sigma^2\vartheta^2N^2-2\sigma^2\vartheta(1-\vartheta)N^2-\sigma^2(1-\vartheta)^2N^2\big]\\
				&+\frac{1}{M_1}\big[2\sigma^2(\vartheta NI_{n+1}^*+(1-\vartheta)NI_n)-2\beta I_n\big]\\
				&+\frac{1}{M_1}\big[-\sigma^2\vartheta^2I_{n+1}^*-2\sigma^2\vartheta(1-\vartheta)I_nI_{n+1}^*-\sigma^2(1-\vartheta)^2I_n^2\big]
				<0.
			\end{aligned}
			$$
			By the large law of martingales in \cite{36}, $\lim_{n\rightarrow\infty}\frac{S_n}{n+1}=0$, Therefore,
			$$
			\begin{aligned}
				\lim_{n\rightarrow\infty} \sup \frac{\log I_{n+1}}{(n+1)h}=\lim_{n\rightarrow\infty} \sup \frac{\log I_0+S_n+\sum_{i=0}^{n}\big(f(I_i)h+O(h^2)\big)}{(n+1)h}<0.
			\end{aligned}
			$$
		\end{proof}

	\begin{theorem}\label{thm5}{\rm
		If $\sigma^2>\max\{\frac{\beta}{N}, \frac{\beta^2}{2(\mu+\gamma)},\frac{\beta^2}{8(\mu+\gamma)-4\beta N}\}$, there  exists  $h_0$  s.t.  
		$
		\lim_{n\rightarrow\infty} \sup \frac{\log I_{n+1}}{(n+1)h}<0
		$
		for $0<h<h_0$, $i.e.$ $I_{n+1}\rightarrow0$.}
\end{theorem}
		\begin{proof}  Denote
			$$
			f(I_n)=\frac{1}{M_1}\{2(\beta N-\mu-\gamma-\beta I_n)-\sigma^2\big[\vartheta(N-I_{n+1}^*)-(1-\vartheta)(N-I_n)\big]^2\}.
			$$
			$\mathbf{(i)}$ If $I_n>I_{n+1}^*$, then
			$$
			\begin{aligned}
				f(I_n)<&2(\beta N-\mu-\gamma-\beta I_n)-\sigma^2(N-I_n)^2\\
				=&-\sigma^2I_n^2+(2N\sigma^2-2\beta)I_n+(2\beta N-\sigma^2N^2-2\mu-2\gamma).
			\end{aligned}
			$$
			When \( I_n=\frac{N\sigma^2-\beta}{\sigma^2}>0 \) $i.e$. $\sigma^2>\frac{\beta}{N}$, it holds that
			$$
			\begin{aligned}
				\max f(I_n)=\frac{4\sigma^2(2\beta N-\sigma^2N^2-2\mu-2\gamma)+(2N\sigma^2-2\beta)^2}{4\sigma^2}
				=\frac{\beta^2}{\sigma^2}-2\mu-2\gamma 
				< 0.
			\end{aligned}
			$$
			$\mathbf{(ii)}$ If $I_n<I_{n+1}^*$, we have $R_n<2$ when $0<h<\frac{1}{2(\beta N+\mu+\gamma)}$. Hence,
			$$
			\begin{aligned}
				f(I_n)<&2(\beta N-\mu-r-\beta \frac{1}{R_n}I_{n+1}^*)-\sigma^2(N-I_{n+1}^*)^2\\
				<&2(\beta N-\mu-\gamma- \frac{1}{2}\beta I_{n+1}^*)-\sigma^2(N-I_{n+1}^*)^2\\
				=&-\sigma^2{I_{n+1}^*}^2+(2N\sigma^2-\beta)I_{n+1}^*+(2\beta N-\sigma^2N^2-2\mu-2\gamma)\\
				:=&g(I_{n+1}^*).
			\end{aligned}
			$$
			When \( I_{n+1}^*=\frac{2\sigma^2N-\beta}{2\sigma^2}>0 \) $i.e.$ $\sigma^2>\frac{\beta}{2N}$  holds, we have
			$$
			\begin{aligned}
				\max g(I_{n+1}^*)=\frac{4\sigma^2(2\beta N-\sigma^2N^2-2\mu-2\gamma)+(2N\sigma^2-\beta)^2}{4\sigma^2}
				=\beta N-2\mu-2\gamma 
				< 0.
			\end{aligned}
			$$
			\par 
			From $\mathbf{(i)}$ and $\mathbf{(ii)}$, it follows that $f(I_n)<0$; similar to the proof of Theorem \ref{thm4}, we obtain $I_{n+1}\rightarrow0$.
		\end{proof}
	
	\begin{theorem}\label{thm6}{\rm
		If $R_0^S>1$ and $\frac{3-\sqrt{5}}{4}\leq\vartheta<\frac{1}{2}$, there exists $h_0>0$ s.t. 
		$
		\lim_{n\rightarrow\infty} \inf I_{n+1}\leq\xi^+\leq\lim_{n\rightarrow\infty} \sup I_{n+1}$ for $0<h<h_0.
		$}
\end{theorem}
		\begin{proof}
			Let $A=\beta N-\mu-\gamma$. There exists \( h_0^1 > 0 \), when $0<h<h_0^1$ and from the definitions of \( f(x) \) and \( I_{n+1}^* \), it follows that
			$$
			\begin{aligned}
				f(x)=&\frac{1}{M_1}\{2(\beta N-\mu-\gamma-\beta x)-\sigma^2\big[\vartheta(N-I_{n+1}^*)-(1-\vartheta)(N-x)\big]^2\}\\
				=&\frac{1}{M_1}(c_0+c_1x+c_2x^2+c_3x^3+c_4x^4),
			\end{aligned}
			$$
			where
			$$
			\begin{aligned}
				M_1=&2R(h,A_h)\big[1-\theta h(\beta N-\mu-r-\beta I_n)\big]^2\\
				=&2\big[1+(1-\theta)h(\beta N-\mu-\gamma-\beta I_n)\big]^2-2\big[\sqrt{h}\sigma\vartheta(N-I_{n+1}^*)A_h\\
				&+\sqrt{h}\sigma(1-\vartheta)(N-I_n)A_h\big]^2\big[1-\theta h(\beta N-\mu-\gamma-\beta I_n)\big]^2>0,\\
				c_0=&\sigma^2N^2(2\vartheta-1)^2-2A+O(h)+O(h^2),\\
c_1=&2\beta+2\sigma^2\vartheta(1-\vartheta)N-2\sigma^2(1-\vartheta)^2N+2\sigma^2N[\vartheta(1-\vartheta)-\vartheta^2]+O(h)+O(h^2),\\
c_2=&\sigma^2[\vartheta-(1-\vartheta)]^2+O(h)+O(h^2)\geq0,\\
	c_3=&2\beta\sigma^2(1-\theta)[\vartheta(1-\vartheta)-\vartheta^2]h+2\beta\sigma^2\theta [(1-\vartheta)^2-\vartheta(1-\vartheta)]h+O(h^2)\geq0,\\
c_4=&\sigma^2h^2\beta^2[\vartheta(1-\theta)+(1-\vartheta)\theta]^2>0.				.
			\end{aligned}
			$$
			From the definition of \( R_0^S \), it follows that when \( R_0^S > 1 \), we have $2\beta>\sigma^2N$ and $2A>\sigma^2N$. Thus, we have
			$$
			\begin{aligned}
				c_0<&\sigma^2N^2[(2\vartheta-1)^2-1]+O(h)+O(h^2)
				<0,\\
c_1>&\sigma^2N(-4\vartheta^2+6\theta-1)+2\sigma^2N[\vartheta(1-\vartheta)-\vartheta^2]+O(h)+O(h^2)
				\geq0.				
			\end{aligned}
			$$
			In summary, $c_0<0, c_1\geq0, c_2\geq0,c_3\geq0, c_4>0$. It is obvious that
			$$
			f'(x)=\frac{1}{M_1}(c_1+2c_2x^2+3c_3x^3+4c_4x^4)>0,\,  \, x\in(0,N).
			$$
			Thus, \( f(x) \) is strictly increasing on \( (0,N) \). There exists $h_0^2>0$ such that for $0<h<h_0^2$, we have
			$$
			\begin{aligned}
				f(0)=c_0<0,\quad
				f(N)=\sigma^2[\vartheta(N-\frac{1+(1-\theta)(\mu+\gamma)h}{1+\theta h(\mu+\gamma)}N)]+2(\mu+\gamma)>0.
			\end{aligned}
			$$
			Obviously, \( f(x) \) has a unique positive root in the interval \( (0, N) \), denoted by $\xi^+$. In the following, we prove \( \lim_{n\rightarrow\infty} \inf I_{n+1}\leq\xi^+ \) and \( \lim_{n\rightarrow\infty} \sup I_{n+1}\geq\xi^+ \) separately.\par 
			$\mathbf{(i)}$ Suppose that \( \lim_{n\rightarrow\infty} \inf I_{n+1}\leq\xi^+ \) does not hold. Then $\mathbb{P}(\Omega_1)>\varepsilon$, where $\Omega_1=\{\lim_{n\rightarrow\infty} \inf I_{n+1}\geq\xi^++2\varepsilon\}$. There exists $m_1(\omega)>0$, when \( n\geq m_1 \) holds, we have \( I_n\geq \xi^++\varepsilon \). Hence,
			$$
			\begin{aligned}
				\lim_{n\rightarrow\infty}\frac{\log I_{n+1}}{(n+1)h}\geq& \lim_{n\rightarrow\infty}\frac{1}{(n+1)h}(\log I_0+S_n+\sum_{i=0}^{m_1}\mathbb{E}(\mathbb{E}(\log{R_n}|\mathcal{F}_{t_n})) +f(\xi^++\varepsilon)(n-m_1)h+O(h^2))\\
				\geq&f(\xi^++\varepsilon) 
				>  0.
			\end{aligned}
			$$
			Then, we have \( \log I_{n+1}\rightarrow+\infty\) i.e. \( I_{n+1}\rightarrow+\infty \), which is a contradiction.\par 
			$\mathbf{(ii)}$ Suppose that \( \lim_{n\rightarrow\infty} \sup I_{n+1}\geq\xi^+ \) does not hold. Then, there exists $\Omega_2=\{\lim_{n\rightarrow\infty} \sup I_{n+1}\leq\xi^+-2\varepsilon\}$ such that $\mathbb{P}(\Omega_2)>\varepsilon$. Similarly, there exists $m_2(\omega)>0$, we have \( I_n\leq \xi^+-\varepsilon \) when \( n\geq m_2 \). It follows that
			$$
			\begin{aligned}
				\lim_{n\rightarrow\infty}\frac{\log I_{n+1}}{(n+1)h}\leq& \lim_{n\rightarrow\infty}\frac{1}{(n+1)h}(\log I_0+S_n+\sum_{i=0}^{m_2}\mathbb{E}(\mathbb{E}(\log{R_n}|\mathcal{F}_{t_n})) +f(\xi^+-\varepsilon)(n-m_2)h+O(h^2))\\
				\leq&f(\xi^+-\varepsilon) 
				<  0.
			\end{aligned}
			$$
			Thus, we have \( \log I_{n+1}\rightarrow-\infty\) i.e. \( I_{n+1}\rightarrow0 \), which is a contradiction.\par 
			From $\mathbf{(i)}$ and $\mathbf{(ii)}$, it is clear that $\lim_{n\rightarrow\infty} \inf I_{n+1}\leq\xi^+\leq\lim_{n\rightarrow\infty} \sup I_{n+1}.$
		\end{proof}

\section{Numerical Simulation}\label{s4}

In this part, we conduct numerical simulations through a series of examples, providing visual representations of the theoretical results in figures. Example \ref{exam1} verifies that the Wiener truncation ensures the positivity of the numerical solution. Example \ref{exam2} compares the computation time and error between Eq. (\ref{ssbe2}) and Eq. (\ref{In4}). Finally, Examples \ref{exam3}-\ref{exam5} validate the extinction and persistence of the disease under three different conditions.

\begin{example}\label{exam1}(Effect of Wiener truncation on positivity preservation) {\rm 
We first examine the role of Wiener truncation in ensuring the positivity of the numerical solution, a key requirement for epidemic models.  The parameters are fixed as \( N=100, \mu=20, \gamma=25, \theta=0.5, \vartheta=0.5, I_0=50 \), and two different combinations of \(\sigma\) and \(h\) are tested. 

\begin{center}
\includegraphics[width=0.40\linewidth]{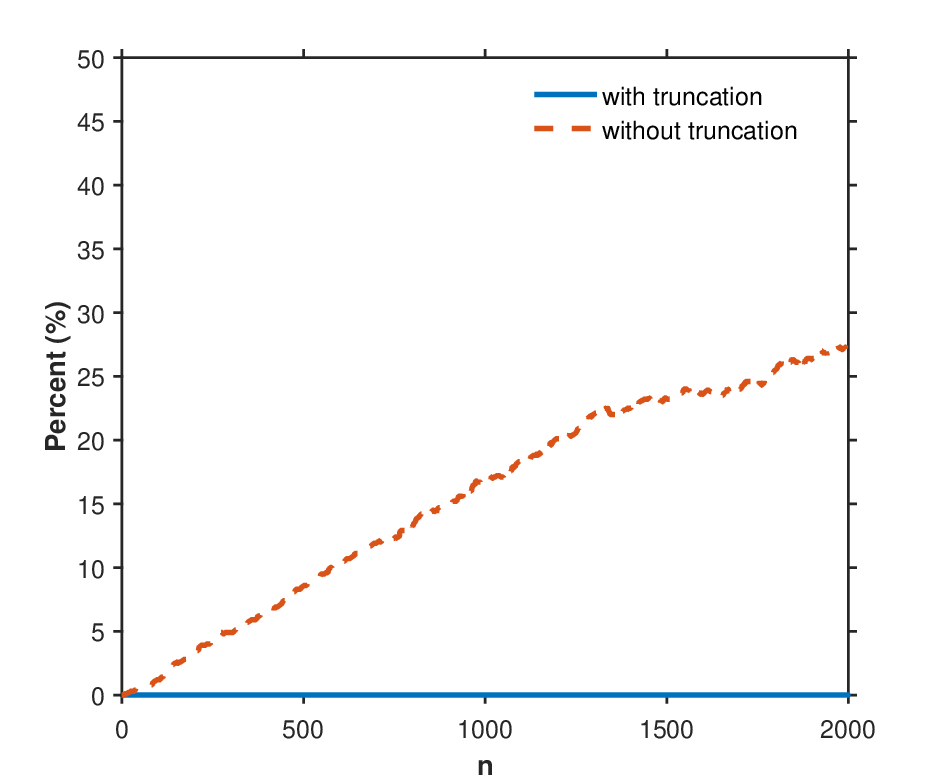} 
\includegraphics[width=0.40\linewidth]{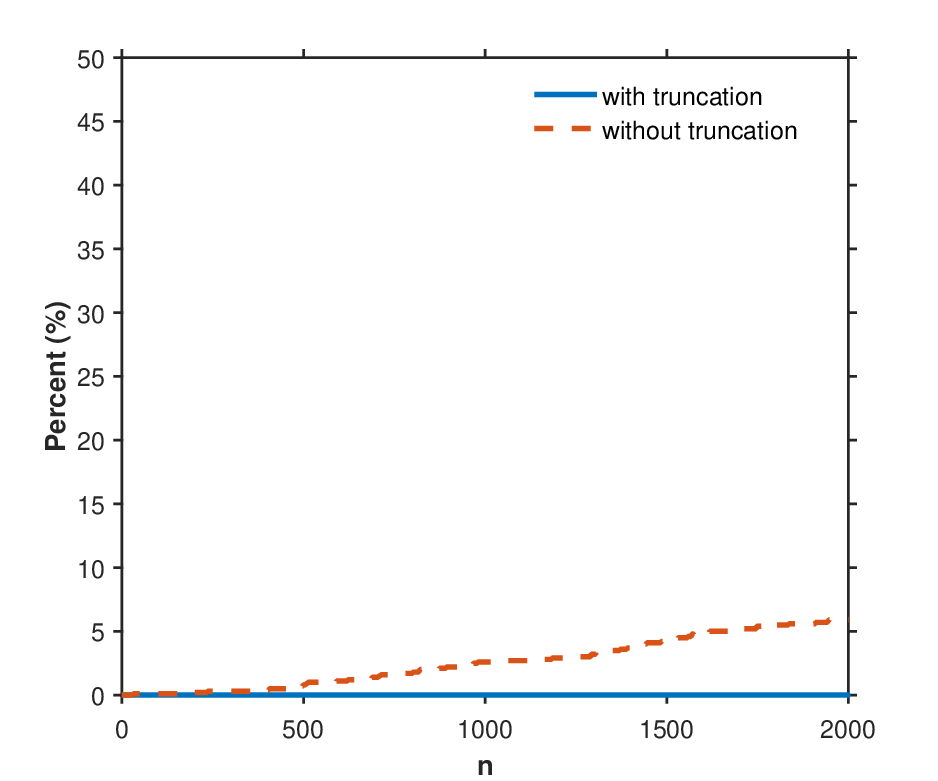}

{\rm Figure 1. The influence of the Wiener truncation on the model solution:  (a) $\sigma=0.03, h=0.01$ and (b) $\sigma=0.08, h=0.001$. }
\end{center}

Figure 1 compares the probability of obtaining non-physical negative values due to infection with and without Wiener truncation under two distinct parameter sets.In the absence of truncation, a non-negligible probability of negative solutions is observed, particularly when the noise intensity or step size is large. By contrast, the truncated version completely eliminates such unphysical outcomes. This demonstrates that the Wiener truncation not only preserves positivity but also enhances the robustness of the numerical scheme, especially in long-time simulations.}
\end{example}

\begin{example}\label{exam2} (Computational efficiency and accuracy of the linearized scheme) {\rm 
One of the main contributions of this work is the linearization of the fully implicit scheme, which reduces computational cost while maintaining accuracy. To quantify this advantage, we compare the nonlinear scheme Eq.(\ref{ssbe2}) and the linearized scheme Eq.(\ref{In4}) in terms of computation time and mean-square error. The parameter values are set as follows: \( N=1000, \beta=0.5/N, \mu=0.02, \gamma=0.1, \sigma=0.3, \theta=\vartheta=0.5, I_0=50 \). Five different step sizes, ranging from \(h = 0.1\) to \(h = 0.005\), are tested.

\begin{center}
\includegraphics[width=0.40\linewidth]{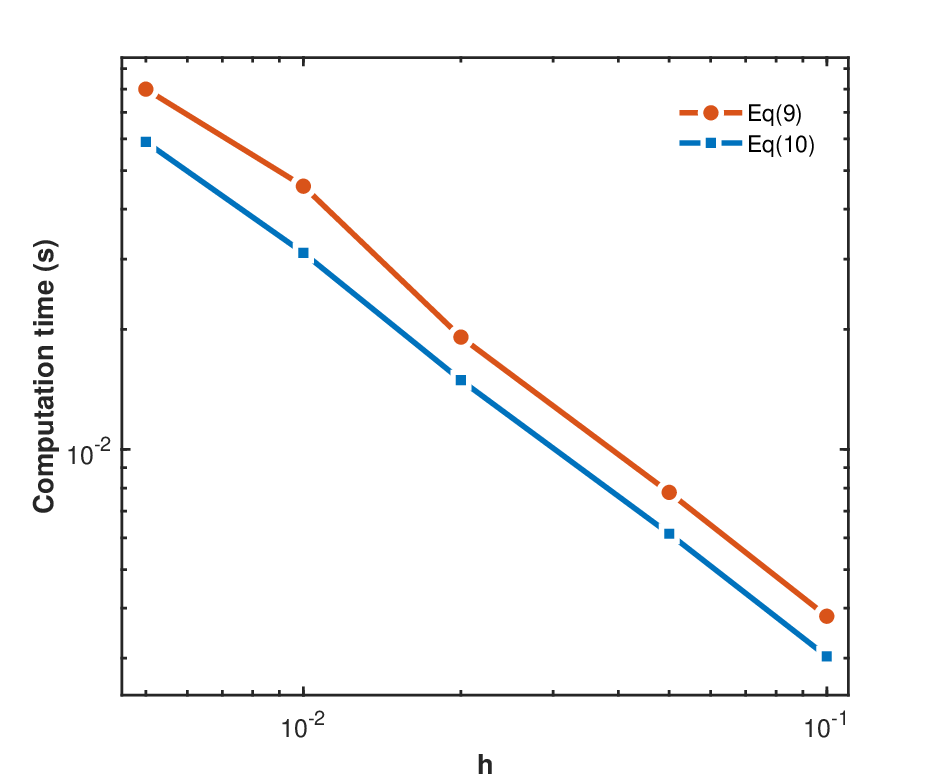}
\includegraphics[width=0.40\linewidth]{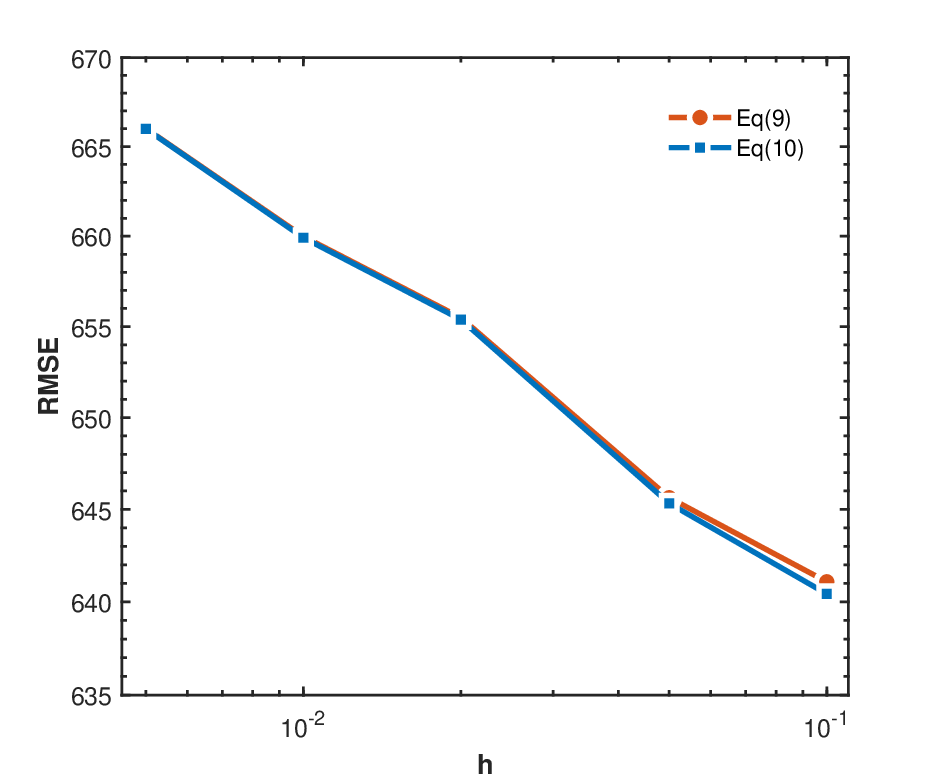}

{\rm Figure 2. Comparison of computation time and error between Eq. (\ref{ssbe2}) and Eq. (\ref{In4}) under different step sizes. }
\end{center}

\begin{table}[!H]
	\centering
	\caption{Comparison of computation time and error under different step sizes (Keep four decimal places).}
	\label{tab:step_comparison}
	\begin{tabular}{
			l
			S[table-format=1.2e-2]
			S[table-format=1.2e-2]
			S[table-format=3.4]
			S[table-format=3.4]
		}
		\toprule
		{Step size $h$} & 
		{\makecell{Computation time  for Eq. (\ref{ssbe2}) [s]}} & 
		{\makecell{Computation time  for Eq. (\ref{In4}) [s]}} & 
		{\makecell{Original   error}} & 
		{\makecell{Linear   error}} \\
		\midrule
		$h=0.100$ & 0.0038 & 0.0030 & 641.0831 & 640.4289 \\
		$h=0.050$ & 0.0078 & 0.0061 & 645.6351 & 645.3278 \\
		$h=0.020$ & 0.0191 & 0.0149 & 655.4887 & 655.3884 \\
		$h=0.010$ & 0.0457 & 0.0311 & 659.9866 & 659.9178\\
		$h=0.005$ & 0.0800 & 0.0590 & 666.0159 & 665.9950 \\
		\bottomrule
	\end{tabular}
\end{table}

 Figure 2 and Table \ref{tab:step_comparison} summarize the results: the linearized scheme is substantially faster while maintaining comparable accuracy. For all step sizes, the linearized scheme Eq.(\ref{In4}) reduces computation time by approximately 25\% compared to the nonlinear scheme Eq.(\ref{ssbe2}). Meanwhile, the mean-square error of the linearized scheme remains very close to that of the original nonlinear scheme, with the ratio between them close to 1. This indicates that the linearization introduces only a negligible loss of accuracy while significantly improving computational efficiency, making scheme Eq.(\ref{In4}) preferable for large-scale or long-term simulations.}

\end{example}

\begin{example}\label{exam3}(Disease extinction under low noise and \(R_0^S < 1\)) {\rm 
 Theorem \ref{thm4} shows that the disease goes extinct when \(R_0^S < 1\) and the noise intensity is sufficiently small. The parameters are chosen as \( N=100, \beta=0.5, \mu=20, \gamma=25, \sigma=0.035, \theta=\vartheta=0.5, h=0.01 \). For these values, \(R_0^S \approx 0.978 < 1\) and \(\sigma^2 = 0.001225 < \beta/(2N) = 0.0025\), satisfying the conditions of Theorem \ref{thm4}.

\begin{center}
\includegraphics[width=0.45\linewidth]{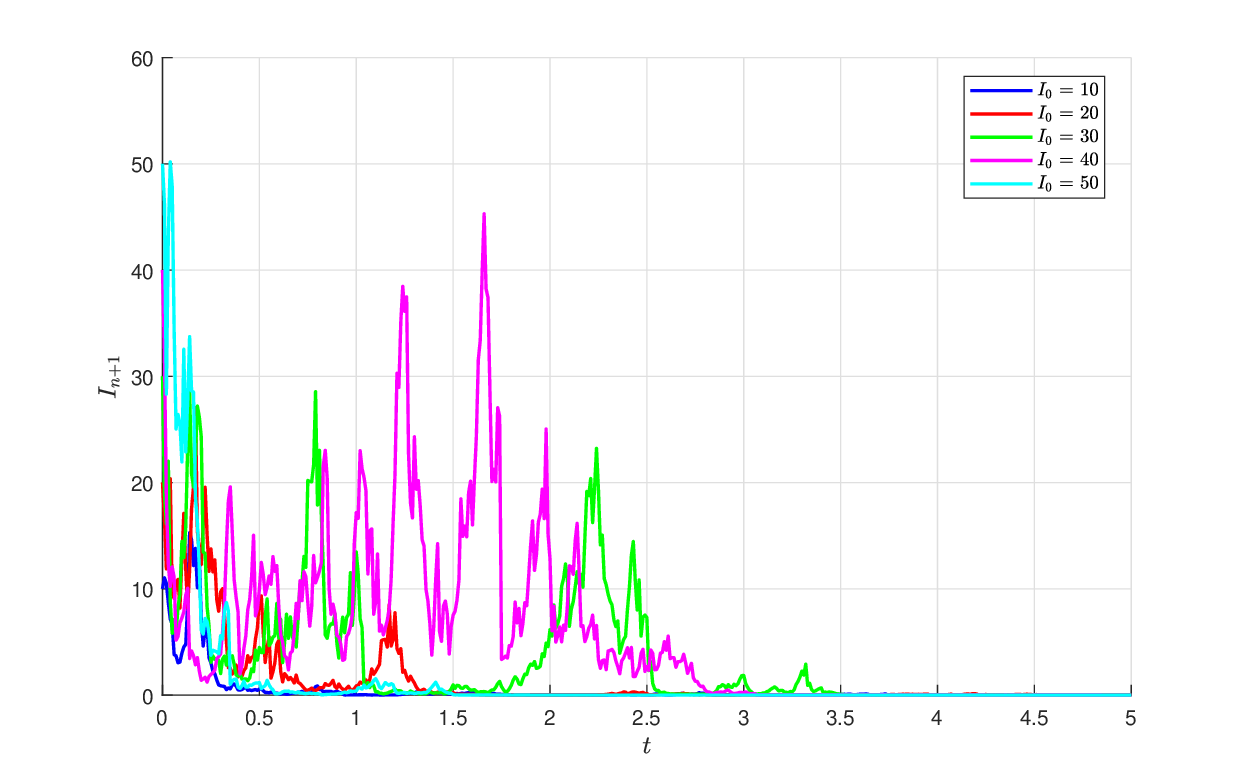} 
\includegraphics[width=0.45\linewidth]{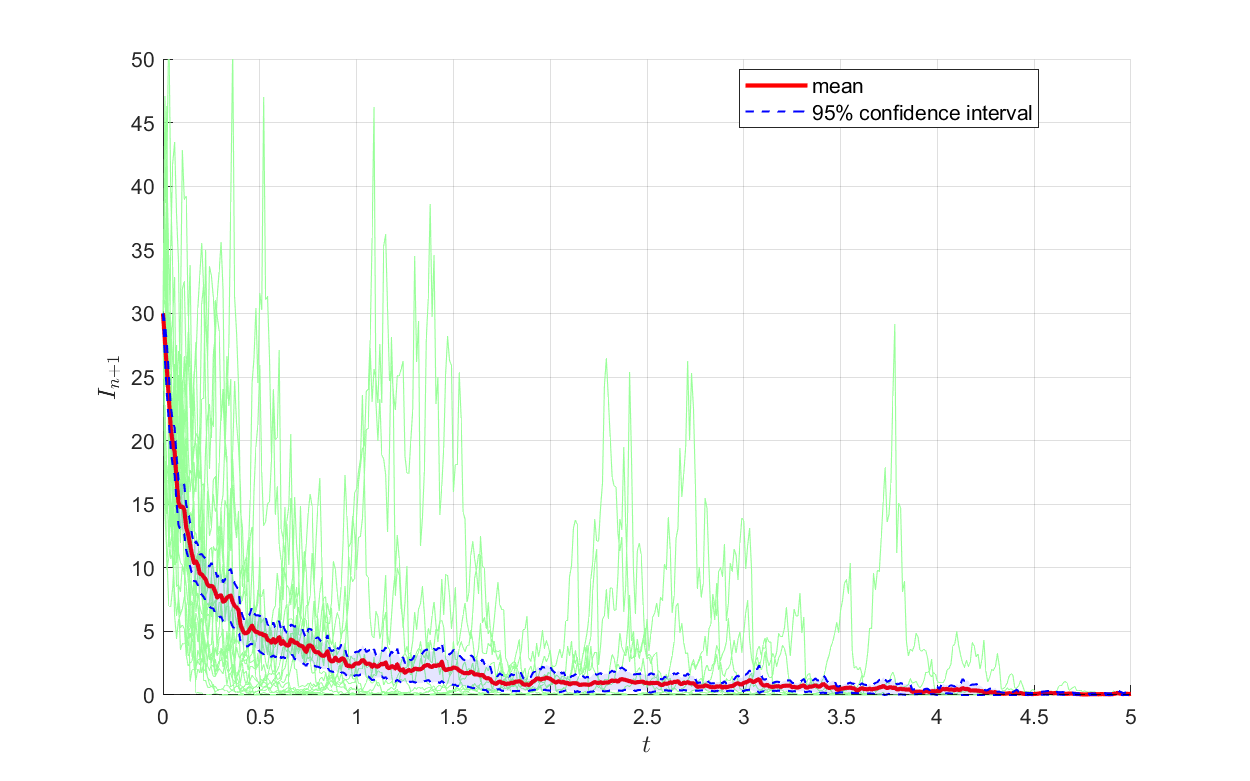}

{\rm Figure 3. Disease extinction: (a) Time evolution for five initial conditions; (b) 100 paths, each showing extinction.}
\end{center}

\begin{center}
	\includegraphics[width=0.9\linewidth]{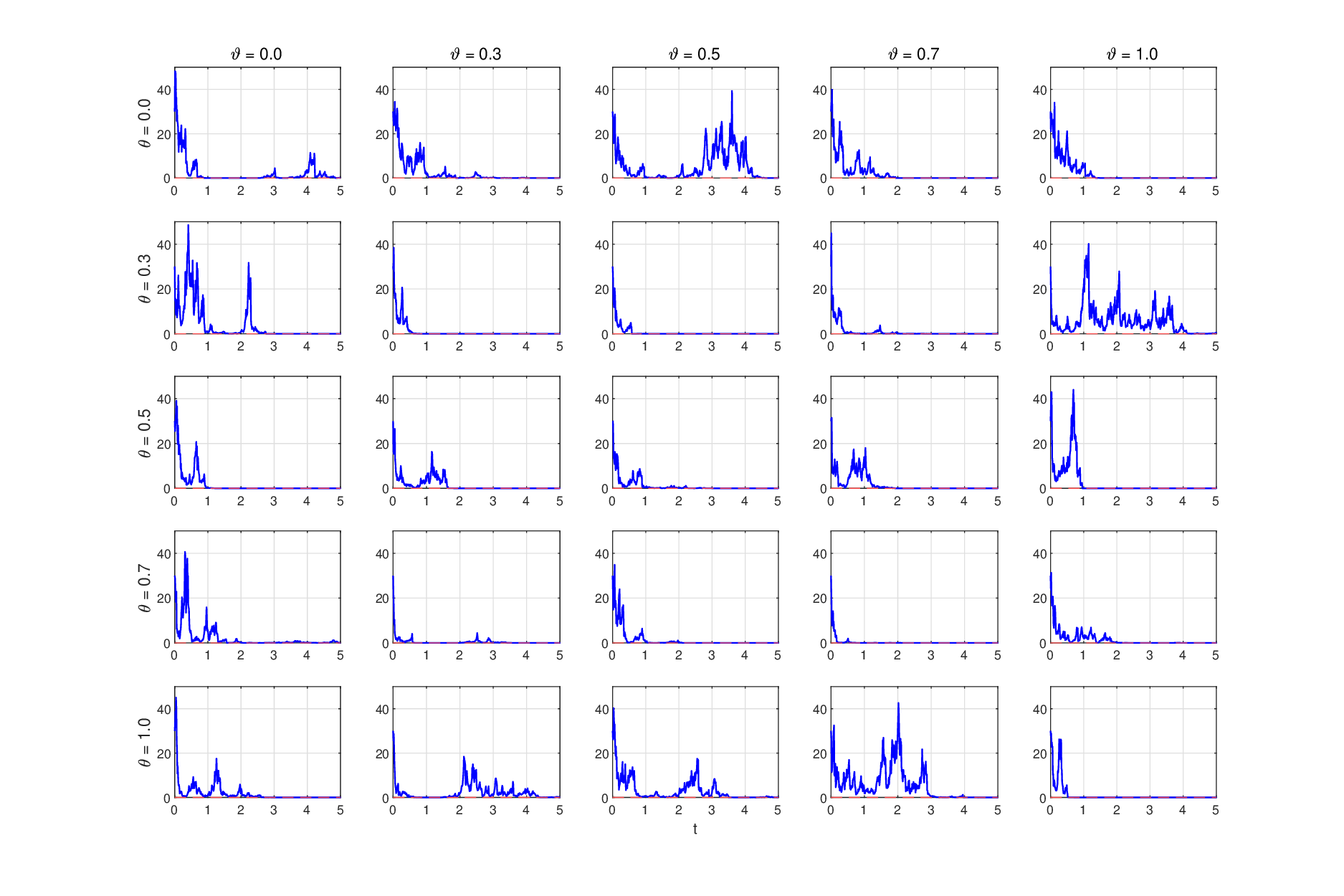}

{\rm Figure 4. Variation of \( I_{n+1} \) with time \( t \) under different values of \( \theta \) and \( \vartheta \). }
\end{center}

Figure 3(a) displays the time evolution of the infected population under five different initial conditions: \(I_0 = 10, 20, 30, 40, 50\). In all cases, the infected population decays monotonically toward zero, confirming extinction. Figure 3(b) presents 100 independent Monte Carlo paths starting from \(I_0 = 30\). Despite stochastic fluctuations, every path converges to zero, indicating that extinction is almost sure.
Figure 4 further examines the sensitivity of this result to the choice of the implicit parameters \(\theta\) and \(\vartheta\). For 25 different combinations,  extinction occurs regardless of the choice of \(\theta\) and \(\vartheta\),  demonstrating the robustness of the theoretical prediction.

}


\end{example}
\begin{example}\label{exam4}
(Disease extinction under high noise) {\rm 
In Theorem \ref{thm5}, sufficiently strong environmental noise can drive the disease to extinction. To test this, we increase the noise intensity to \(\sigma = 0.1\) while keeping all other parameters the same as in Example \ref{exam3}. Here \(\sigma^2 = 0.01 > \beta/N = 0.005\), satisfying the high-noise condition of Theorem \ref{thm5}.

\begin{center}
\includegraphics[width=0.45\linewidth]{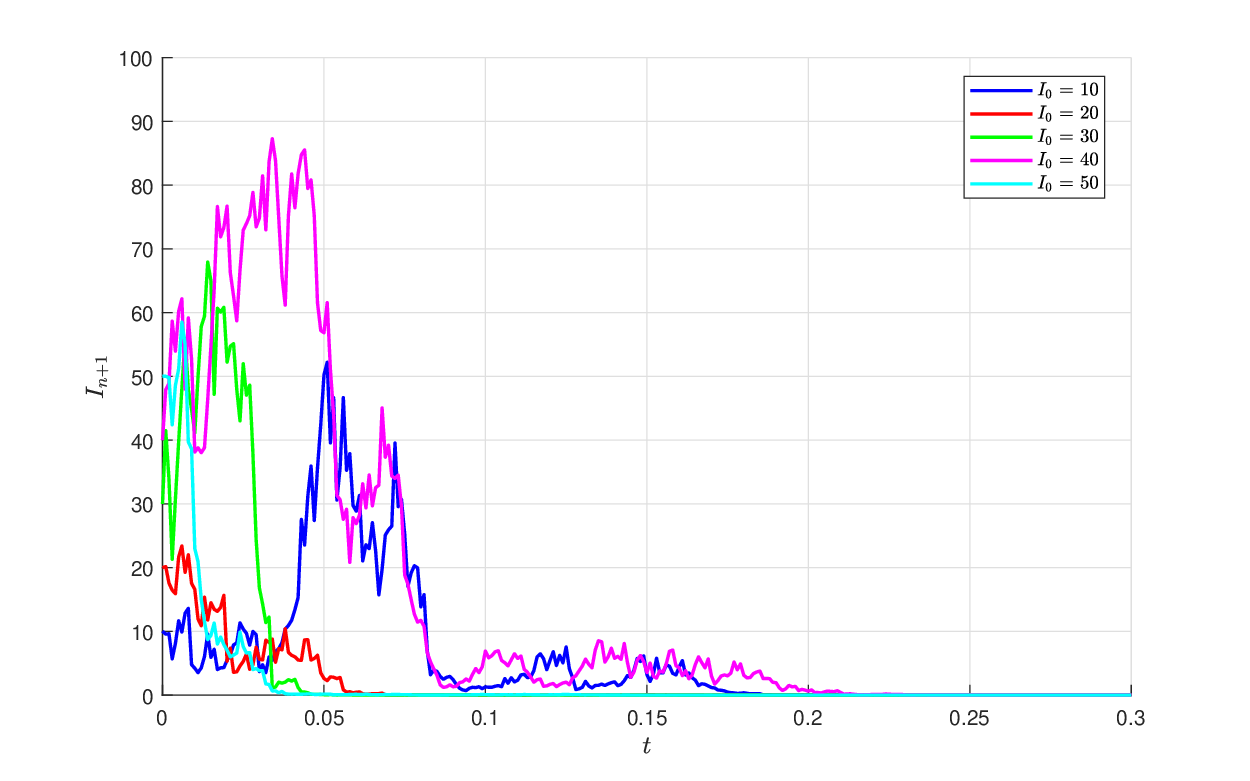} 
\includegraphics[width=0.45\linewidth]{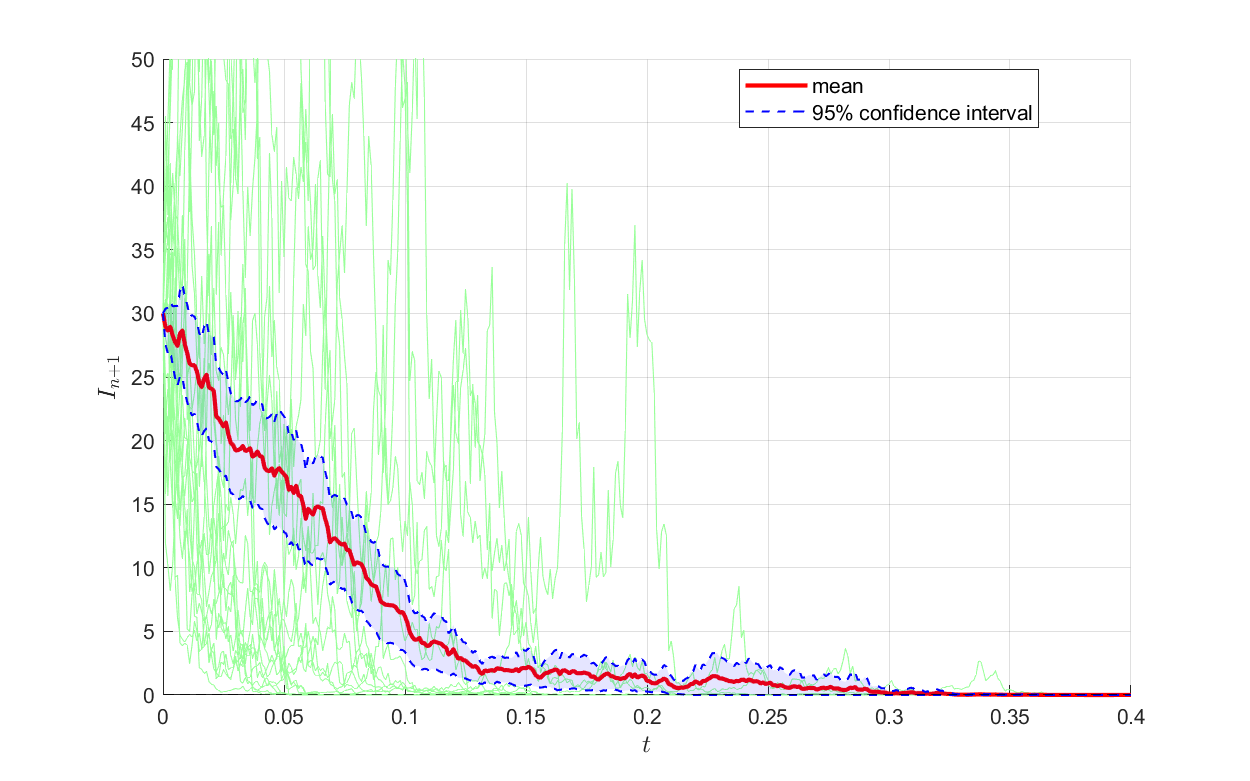}

{\rm Figure 5. Disease extinction: (a) Time evolution for different \( I_0 \) sets; (b) 100  paths.}
\end{center}

\begin{center}
\includegraphics[width=1\linewidth]{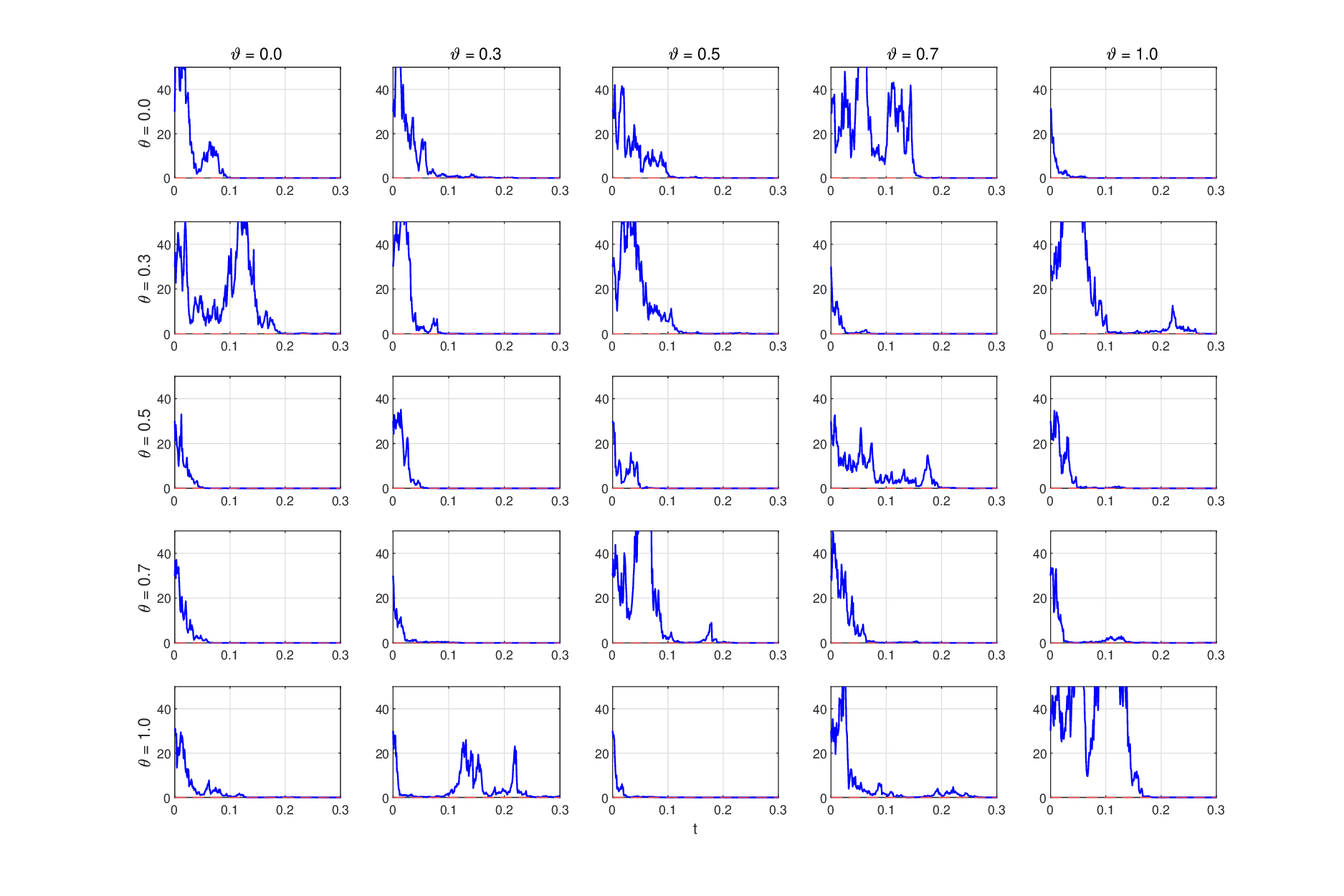}

{\rm Figure 6. The number of infected individuals \( I_{n+1} \) varies with time \( t \) for different values of \( \theta \) and \( \vartheta \).  }
\end{center}

Figure 5 (a) shows that under five different initial conditions, the disease goes extinct in all cases. Figure 5 (b) presents the corresponding Monte Carlo simulation with 100 paths, confirming that extinction is almost sure. Notably, the extinction occurs much faster than in the low-noise case of Example \ref{exam3}, illustrating the accelerating effect of large environmental fluctuations. Figure 6 demonstrates that this high-noise extinction result is also independent of the choice of \(\theta\) and \(\vartheta\), further validating the theoretical analysis.}


\end{example}
\begin{example}\label{exam5} (Disease persistence when \(R_0^S > 1\)) {\rm 
 Theorem \ref{thm6} reveals that the disease persists when \(R_0^S > 1\) and the implicit parameter \(\vartheta\) lies in the interval \([(3-\sqrt{5})/4, 1/2)\). The parameters are set to \( N=100, \beta=0.6, \mu=20, \gamma=25, \sigma=0.035, \theta=0.5, \vartheta=0.4, h=0.01 \). For these values, \(R_0^S \approx 1.178 > 1\) satisfies the required condition.

\begin{center}
\includegraphics[width=0.45\linewidth]{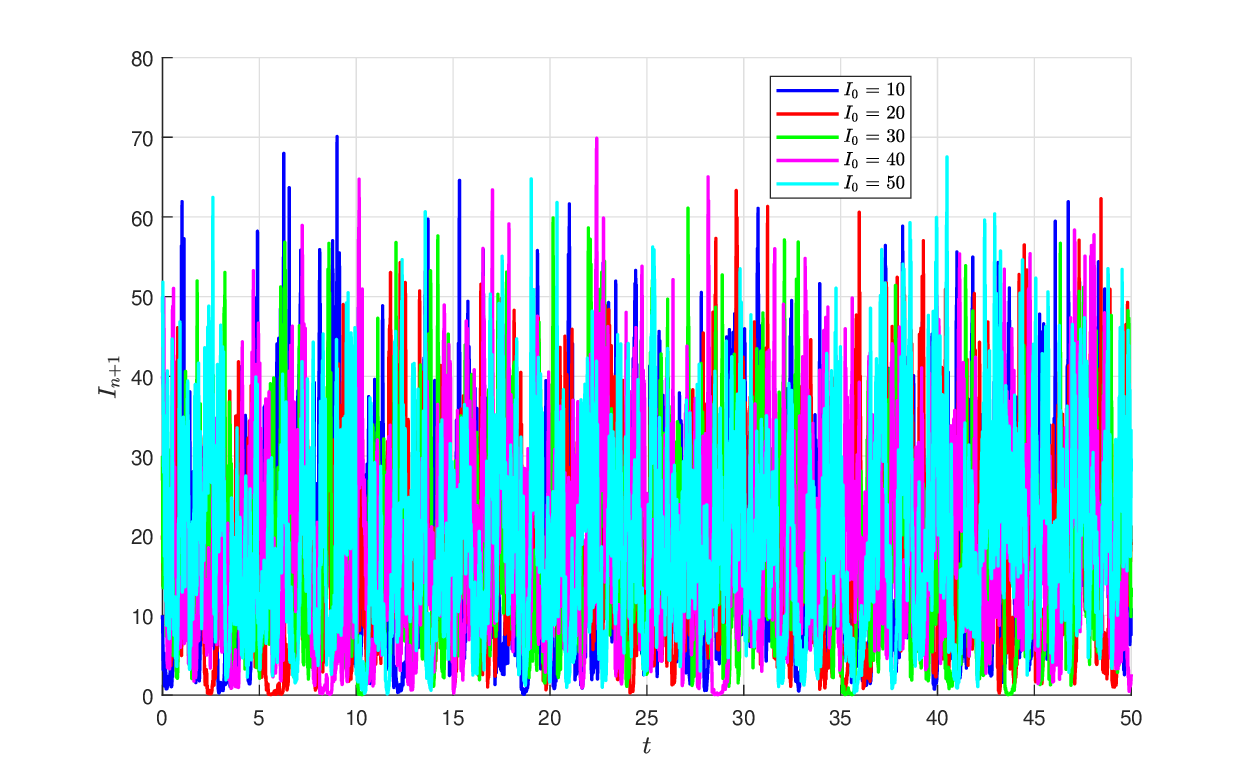} 
\includegraphics[width=0.45\linewidth]{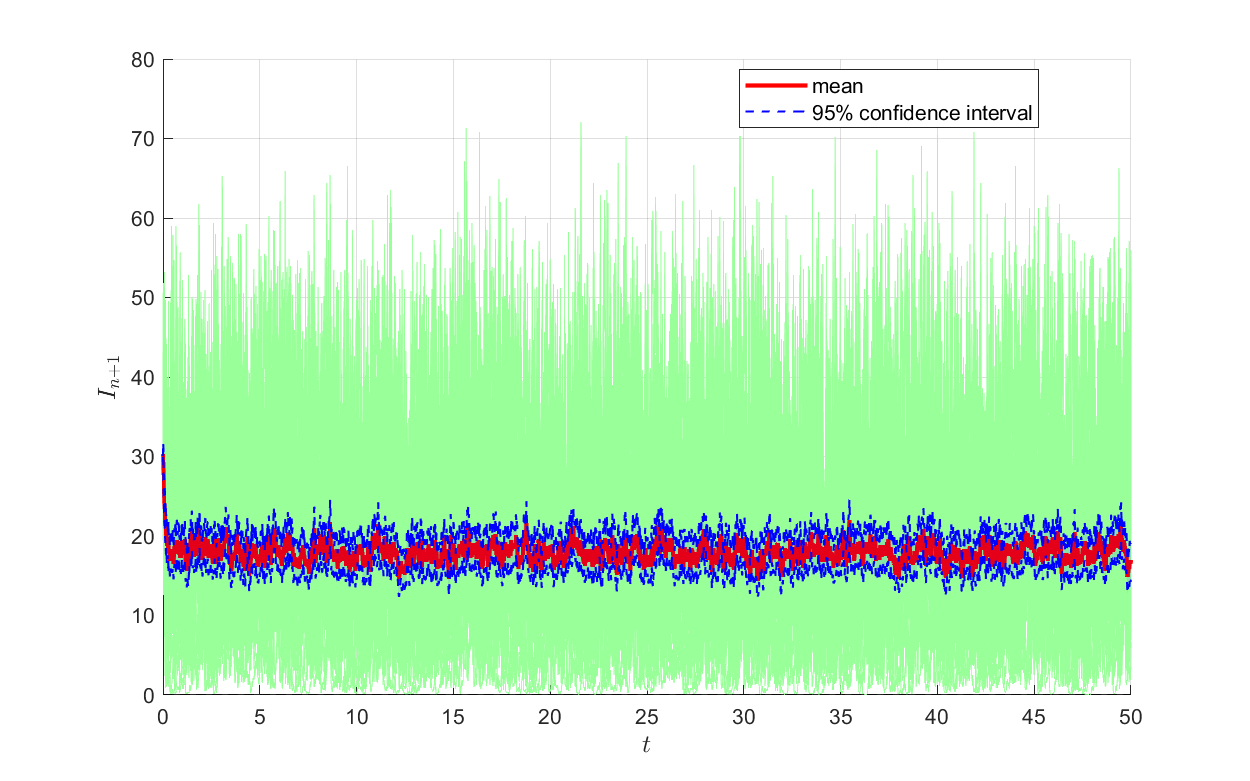}

{\rm Figure 7.    (a) Time evolution for five initial conditions, all converging to a positive endemic level; (b) Mean and 95\% confidence interval from 100  paths, showing that the infected population remains bounded away from zero.}
\end{center}

\begin{center}
\includegraphics[width=1\linewidth]{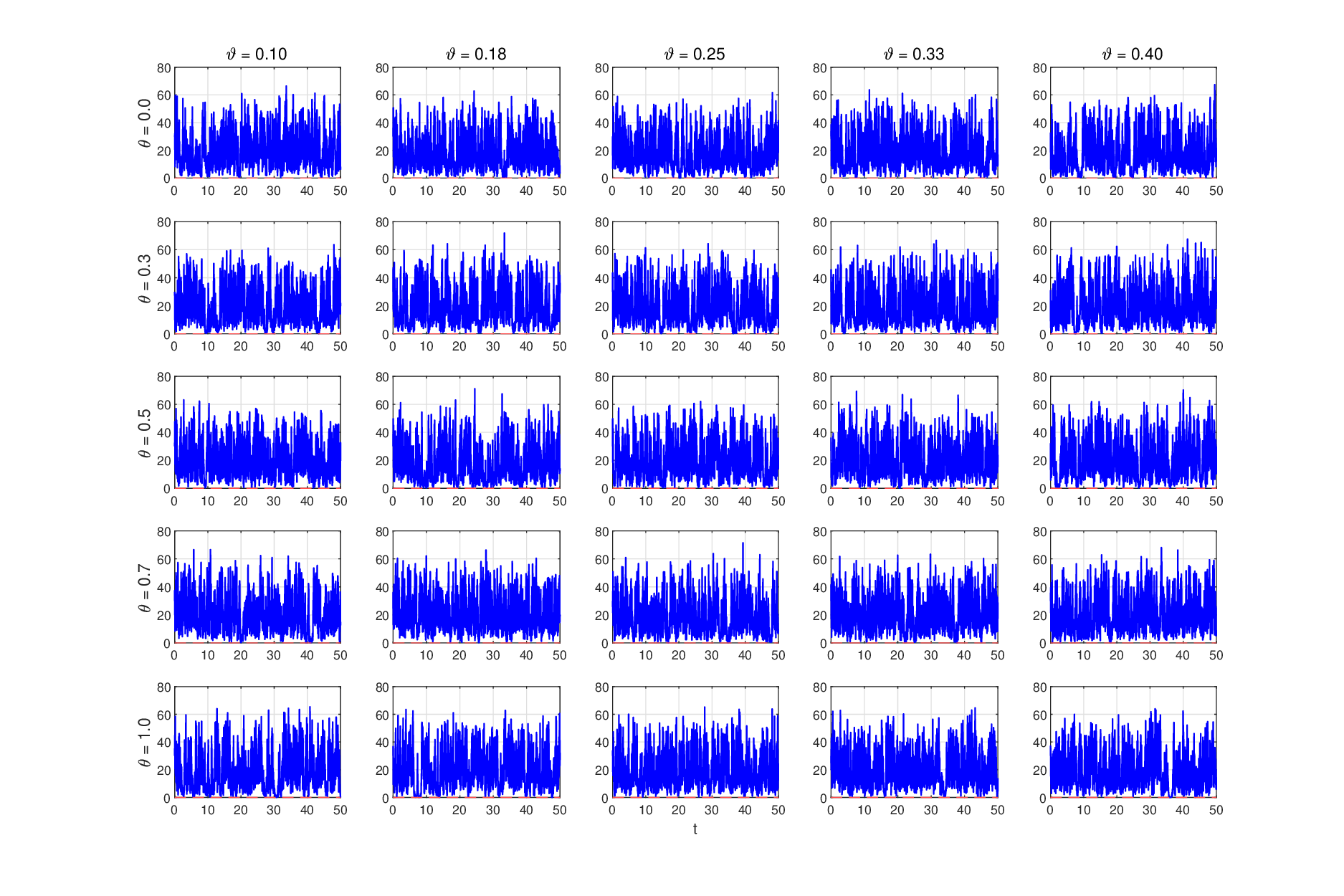}

{\rm Figure 8. Under the conditions of Theorem \ref{thm6}, the trend of \( I_{n+1} \) over time \( t \) for different values of \( \theta \) and \( \vartheta \). }
\end{center}

Figure 7(a) illustrates the long-term behavior of the infected population under five different initial conditions: \(I_0 = 10, 20, 30, 40, 50\). In contrast to the extinction cases, all trajectories converge to a positive endemic level. Figure 7(b) shows the mean and 95\% confidence interval obtained from 100 Monte Carlo paths. The confidence interval remains strictly above zero after the initial transient, providing statistical evidence of persistence. Figure 8 examines the effect of varying \(\theta\) and \(\vartheta\) while keeping other parameters fixed; the disease still persists but with a slightly different endemic level, confirming the robustness of the persistence behavior under the given parameter regime.}


\end{example}
	\section{Conclusion}\label{s5}
	First, by discretizing the stochastic SIS epidemic model (\cite{32}) using the SSBE method (\cite{34}), introducing a Wiener truncation (\cite{33}), and further incorporating the linear implicit $\theta$-methods (\cite{1}) into the diffusion term, a class of fully implicit stochastic SIS models was established by simplifying the model into the linear structure Eq. (\ref{In4}).
	
	Second, in the theoretical proof section, Theorem \ref{thm2} demonstrates the positivity and boundedness of the solution generated by Eq. (\ref{In4}) after introducing the Wiener truncation. Theorem \ref{thm3} proves that the global convergence order of the numerical solution error is $1/2$. Theorems \ref{thm4} to \ref{thm6} discuss the extinction and persistence of the disease under three different scenarios: when the threshold $R_0^S < 1$ in a low-noise environment, the disease eventually goes extinct; high noise also leads to disease extinction; when $R_0^S > 1$ and $(3-\sqrt{5})/4 \leq \vartheta <1/2$, the disease persists.
	
	Finally, numerical simulations are conducted to visually demonstrate the conclusions of the theoretical proofs. Example \ref{exam1} compares the probability of negative solutions occurring with and without truncation, showing that the Wiener truncation ensures the positivity of the model solution. Example \ref{exam2} compares the computation time and error between the nonlinear Eq. (\ref{ssbe2}) and the linear Eq. (\ref{In4}). The results show that the computation time of Eq. (\ref{In4}) is lower than that of Eq. (\ref{ssbe2}), while the difference in error is small, making Eq. (\ref{In4}) a preferable choice for solving. Examples \ref{exam3} and \ref{exam4} verify, under the conditions of Theorem \ref{thm4} and Theorem \ref{thm5}, respectively, that the disease eventually becomes extinct. Example \ref{exam5} demonstrates that under the conditions of Theorem \ref{thm6}, the disease persists.

	\section*{Conflict of interest}
	The authors declare that they have no conflict of interest.
	
	
\section*{Funding}
The work was supported by the Xinjiang Talent Development Fund (XJRC-2025-KJ-PY-KJLJ-108), the National Natural Science Foundation of China (12561047), and the 2025 Central Guidance for Local Science and Technology Development Fund (ZYYD2025ZY20).
	
\section*{ CRediT authorship contribution statement}  Zeyao Jin: Writing – original draft, Methodology, Conceptualization. Zhiming Li: Writing – review \& editing, Supervision, Funding acquisition. Shenglong Chen: Writing – review \& editing.

	\printcredits
	


\end{document}